\documentclass{article}

\usepackage[english]{babel}
\usepackage[latin1]{inputenc}
\usepackage[T1]{fontenc}
\usepackage{graphicx,amsfonts,amsmath,amssymb,amsthm,enumitem,mathrsfs,color,epstopdf,booktabs,subcaption,url}
\usepackage[top=2cm,bottom=3cm,left=2cm,right=2cm]{geometry}
\usepackage{sectsty}\sectionfont{\normalsize}\subsectionfont{\normalsize}

\newtheorem{theorem}{Theorem}[section]

\theoremstyle{definition}
\newtheorem{definition}{Definition}[section]
\newtheorem{remark}{Remark}[section]
\newtheorem{example}{Example}[section]

\numberwithin{equation}{section}
\numberwithin{figure}{section}
\numberwithin{table}{section}

\DeclareMathAlphabet{\mathcal}{OMS}{cmsy}{m}{n}
\DeclareMathAlphabet{\mathscr}{U}{rsfso}{m}{n}

\renewcommand{\epsilon}{\varepsilon}
\newcommand{\xx}{\mathbf{x}}
\newcommand{\yy}{\mathbf{y}}

\begin{document}

\title{Introduction to the theory of generalized locally Toeplitz sequences and its applications}
\author{Carlo Garoni\\
\footnotesize Department of Mathematics, University of Rome Tor Vergata, Rome, Italy (garoni@mat.uniroma2.it)}
\date{}
\maketitle

\begin{abstract}
The theory of generalized locally Toeplitz (GLT) sequences was conceived as an apparatus for computing the spectral distribution of matrices arising from the numerical discretization of differential equations (DEs).
The purpose of this review is to introduce the reader to the theory of GLT sequences and to present some of its applications to the computation of the spectral distribution of DE discretization matrices.
We mainly focus on the applications, whereas the theory is presented in a self-contained tool-kit fashion, without entering into technical details.
The exposition is supposed to be understandable to master's degree students in mathematics. It also discloses new more efficient approaches to the spectral analysis of DE discretization matrices as well as a novel spectral analysis tool that has not been considered in the GLT literature heretofore, i.e., the modulus of integral continuity.

\medskip


\medskip

\end{abstract}

{\footnotesize\tableofcontents}

\section{Introduction}\label{Ch1}

Suppose that we are given a linear differential equation (DE)
\begin{equation}\label{DE}
\mathscr Au=f
\end{equation}
(plus some boundary or initial conditions), where $\mathscr A$ is the linear differential operator, $f$ is the right-hand side, and $u$ is the solution.
Suppose also that the DE \eqref{DE} is discretized by a (linear) numerical method characterized by a mesh fineness parameter~$n$. In this case, the computation of the numerical solution reduces to solving a (square) linear system
\begin{equation}\label{LS}
A_n\mathbf u_n=\mathbf f_n.
\end{equation}
The size of the system \eqref{LS} grows as $n$ increases, i.e., as the mesh is progressively refined with the purpose of obtaining more and more accurate approximations $\mathbf u_n$ converging (in some topology) to the exact solution $u$. We are therefore in the presence of a sequence of matrices $A_n$ such that ${\rm size}(A_n)=d_n\to\infty$ as $n\to\infty$.
What is often observed in practice is that $A_n$ enjoys a spectral distribution as $n\to\infty$, in the sense that, for a large class of test functions~$F$,
\begin{equation*}
\lim_{n\to\infty}\frac1{d_n}\sum_{i=1}^{d_n}F(\lambda_i(A_n))=\frac1{\mu_k(D)}\int_DF(\kappa(\mathbf s)){\rm d}\mathbf s,
\end{equation*}
where $\lambda_i(A_n)$, $i=1,\ldots,d_n$, are the eigenvalues of $A_n$, $\mu_k$ is the Lebesgue measure in $\mathbb R^k$, and $\kappa:D\subset\mathbb R^k\to\mathbb C$ is a function. In this situation, $\kappa$ is referred to as the {\em spectral symbol} of the sequence $\{A_n\}_n$. The information contained in the spectral symbol $\kappa$ can be informally summarized as follows: assuming that $n$ is large enough, the eigenvalues of $A_n$, except possibly for $o(d_n)$ outliers, are approximately equal to the samples of $\kappa$ over a uniform grid in $D$. For example, if $k=1$, $d_n=n$ and $D=[a,b]$, then, assuming we have no outliers, the eigenvalues of $A_n$ are approximately equal to 
\[ \kappa\Bigl(a+i\,\frac{b-a}{n}\Bigr),\qquad i=1,\ldots,n, \]
for $n$ large enough. Similarly, if $k=2$, $d_n=n^2$ and $D=[a_1,b_1]\times [a_2,b_2]$, then, assuming we have no outliers, the eigenvalues of $A_n$ are approximately equal to
\[ \kappa\Bigl(a_1+i_1\,\frac{b_1-a_1}n,\,\,a_2+i_2\,\frac{b_2-a_2}n\Bigr),\qquad i_1,i_2=1,\ldots,n, \]
for $n$ large enough. It is therefore clear that the spectral symbol $\kappa$ provides a ``compact'' and quite accurate description of the spectrum of the matrices $A_n$ (for $n$ large enough). For a list of practical uses of the spectral symbol, we refer the reader to \cite[Section~1.1]{GLT-bookI}; see also \cite[p.~617]{F-test}.

The theory of generalized locally Toeplitz (GLT) sequences was conceived as an apparatus 
for computing the spectral symbol~$\kappa$ of sequences of matrices $\{A_n\}_n$ arising from DE discretizations. Indeed, every sequence of this kind often turns out to be a GLT sequence. This happens, for instance, if the numerical method used to discretize the DE under consideration 
belongs to the family of the so-called ``local methods''. Local methods are, for example, finite difference (FD) methods, finite element (FE) methods with ``locally supported'' basis functions, isogeometric analysis (IgA), collocation methods, etc.; in short, all standard numerical methods for the approximation of DEs.

The present review was born from the need to introduce, in a relatively simple and effective way, the theory of GLT sequences and its applications to doctoral students and researchers who are new to the subject. This need has not been satisfied by the GLT literature heretofore.
The material included herein consists of the lecture notes used by the author in his PhD courses on GLT sequences held in various universities.
The exposition is supposed to be understandable to master's degree students in mathematics and includes examples born from students' questions.
With respect to the book~\cite{GLT-bookI}, besides the advantage of being accessible to a wider audience, the present review also contains a number of updates, which reflect the main advances in the field of GLT sequences that followed the publication of \cite{GLT-bookI}. The major updates that are worth to be emphasized are the following.
\begin{itemize}[nolistsep,leftmargin=*]
	\item Several new and more efficient approaches to the GLT spectral analysis of DE discretization matrices are proposed on the basis of the results appeared in \cite{a.u.,NLAA}. 
	\item A novel tool for the GLT spectral analysis of DE discretization matrices is introduced for the first time, namely, the modulus of integral continuity; see Section~\ref{moduli}. This tool is used in the context of FE discretizations to obtain spectral distribution results under minimal integrability assumptions on the DE coefficients; see Section~\ref{sec:FEs} (in particular, Theorem~\ref{FE_T1}).
\end{itemize}
The review is organized as follows. In Section~\ref{Ch9} we present the theory of GLT sequences in a self-contained tool-kit fashion, without entering into technical details.
In Section~\ref{Ch10} we present a number of applications of the theory of GLT sequences to the computation of the spectral distribution of DE discretization matrices.
Following~\cite{GLT-bookI}, we here consider only unidimensional applications (i.e., unidimensional DEs) both for simplicity and because the key ``GLT ideas'' are better conveyed in the univariate setting, without the ``fog'' of technical details. 
For the multivariate setting, we refer the reader to \cite{aip} for a quick overview and to \cite{reducedGLT,rectangularGLT,mblockGLT,GLT-bookII} for more advanced studies; see also \cite{blockGLT} for the so-called univariate ``block'' case.


\section{The Theory of GLT Sequences: A Summary}\label{Ch9}

In this section we present a {\em self-contained} summary of the theory of GLT sequences. Despite its conciseness, our presentation contains {\em everything one needs to know} in order to understand the applications presented in Section~\ref{Ch10}.

\paragraph{Matrix-sequences}
A {\em matrix-sequence} is a sequence of the form $\{A_n\}_n$, where $A_n$ is an $n\times n$ matrix. 

\paragraph{Singular value and eigenvalue distribution of a matrix-sequence}
Let $\mu_k$ be the Lebesgue measure in $\mathbb R^k$. Throughout this work, all terminology from measure theory (such as ``measurable set'', ``measurable function'', ``a.e.'', etc.) is always referred to the Lebesgue measure.
Let $C_c(\mathbb R)$ (resp., $C_c(\mathbb C)$) be the space of continuous complex-valued functions with bounded support defined on $\mathbb R$ (resp., $\mathbb C$). If $A\in\mathbb C^{n\times n}$, the singular values and eigenvalues of $A$ are denoted by $\sigma_1(A),\ldots,\sigma_n(A)$ and $\lambda_1(A),\ldots,\lambda_n(A)$, respectively. The set of the eigenvalues (i.e., the spectrum) of $A$ is denoted by $\Lambda(A)$.

\begin{definition}[\textbf{singular value and eigenvalue distribution of a matrix-sequence}]\label{def:distr}
Let $\{A_n\}_n$ be a matrix-sequence and let $f:D\subset\mathbb R^k\to\mathbb C$ be a measurable function defined on a set $D$ with $0<\mu_k(D)<\infty$.
\begin{itemize}[nolistsep,leftmargin=*]
	\item We say that $\{A_n\}_n$ has a singular value distribution described by $f$, and we write $\{A_n\}_n\sim_\sigma f$, if
	\begin{equation}\label{tilde_sigma}
	\lim_{n\to\infty}\frac1{n}\sum_{i=1}^{n}F(\sigma_i(A_n))=\frac1{\mu_k(D)}\int_DF(|f(\xx)|){\rm d}\xx,\qquad\forall\,F\in C_c(\mathbb R).
	\end{equation}
	In this case, $f$ is called the {\em singular value symbol} of $\{A_n\}_n$.
	\item We say that $\{A_n\}_n$ has a spectral (or eigenvalue) distribution described by $f$, and we write $\{A_n\}_n\sim_\lambda f$, if
	\begin{equation}\label{tilde_lambda}
	\lim_{n\to\infty}\frac1{n}\sum_{i=1}^{n}F(\lambda_i(A_n))=\frac1{\mu_k(D)}\int_DF(f(\xx)){\rm d}\xx,\qquad\forall\,F\in C_c(\mathbb C).
	\end{equation}
	In this case, $f$ is called the {\em spectral (or eigenvalue) symbol} of $\{A_n\}_n$.
\end{itemize}
\end{definition}

Whenever we write a relation such as $\{A_n\}_n\sim_\sigma f$ or $\{A_n\}_n\sim_\lambda f$, it is understood that $\{A_n\}_n$ and $f$ are as in Definition~\ref{def:distr}, i.e., $\{A_n\}_n$ is a matrix-sequence and $f$ is a measurable function defined on a subset $D$ of some $\mathbb R^k$ with $0<\mu_k(D)<\infty$. 
If $\{A_n\}_n$ has both a singular value and a spectral distribution described by $f$, we write $\{A_n\}_n\sim_{\sigma,\lambda}f$.

We report in {\bf S1} and {\bf S2} the statements of two useful results concerning the spectral distribution of matrix-sequences.
For the proof of \textbf{S1}, see \cite[Lemma~2.1 and Corollary~3.1]{GLT-bookI}. For the proof of {\bf S2}, see \cite{NLAA}.
Throughout this work, if $A\in\mathbb C^{n\times n}$ and $1\le p\le\infty$, we denote by $\|A\|_p$ the Schatten $p$-norm of $A$, i.e., the $p$-norm of the vector $(\sigma_1(A),\ldots,\sigma_n(A))$ formed by the singular values of $A$.
The Schatten 1-norm $\|A\|_1$ is the sum of the singular values of $A$ and is often referred to as the trace-norm of $A$. The Schatten $2$-norm $\|A\|_2$ coincides with the Frobenius norm. Indeed, since the squares of the singular values of $A$ are the eigenvalues of $A^*A$, we have
\begin{align*}
\|A\|_2&=\sqrt{\sum_{i=1}^n\sigma_i(A)^2}=\sqrt{\sum_{i=1}^n\lambda_i(A^*A)}=\sqrt{{\rm trace}(A^*A)}=\sqrt{\sum_{j=1}^n(A^*A)_{jj}}\\
&=\sqrt{\sum_{j=1}^n\sum_{i=1}^n(A^*)_{ji}A_{ij}}=\sqrt{\sum_{i,j=1}^n|A_{ij}|^2}.
\end{align*}
The Schatten $\infty$-norm $\|A\|_\infty$ is the largest singular value of $A$ and coincides with the induced 2-norm $\|A\|$, also known as the spectral or Euclidean norm; see \cite[p.~33]{GLT-bookI}. For more on Schatten $p$-norms, see \cite[Section~2.4.3]{GLT-bookI} and \cite{Bhatia}.
The (topological) closure of a set $S$ is denoted by $\overline S$.
We use a notation borrowed from probability theory to indicate sets. For example, if $f,g:D\subseteq\mathbb R^k\to\mathbb R$, then $\{f\le1\}=\{\xx\in D:f(\xx)\le1\}$, $\mu_k\{f>0,\:g<0\}$ is the measure of the set $\{\xx\in D:f(\xx)>0,\:g(\xx)<0\}$, etc. 
\begin{enumerate}[leftmargin=18pt,nolistsep]
	\item[\textbf{S1.}] If $\{A_n\}_n\sim_\lambda f$ and $\Lambda(A_n)\subseteq S$ for all $n$ then 
	$f\in\overline{S}$ a.e.
	\item[\textbf{S2.}] If $A_n=X_n+Y_n$ where
	\begin{itemize}[nolistsep,leftmargin=*]
		\item each $X_n$ is Hermitian and $\{X_n\}_n\sim_\lambda f$,
		\item $\|Y_n\|_2=o(n^{1/2})$,
	\end{itemize}
	then $\{A_n\}_n\sim_\lambda f$.
\end{enumerate}

\paragraph{Informal meaning}
Assuming that $f$ is at least continuous a.e., the spectral distribution $\{A_n\}_n\sim_\lambda f$ in \eqref{tilde_lambda} has the following informal meaning: all the eigenvalues of $A_n$, except possibly for $o(n)$ outliers, are approximately equal to the samples of $f$ over a uniform grid in $D$ (for $n$ large enough).\,\footnote{\,Note that, if the eigenvalues of $A_n$ were exact samples of $f$ over a uniform grid in $D$, then \eqref{tilde_lambda} would be intuitively satisfied; think, for example, to the case where $D=[a,b]$ and $\lambda_i(A_n)=f(a+i\,\frac{b-a}{n})$ for $i=1,\ldots,n$.}
We refer the reader to \cite{sdau,F-test} for a mathematical proof under certain assumptions of the informal meaning given here.
For instance, if $k=1$ and $D=[a,b]$, then, assuming we have no outliers, the eigenvalues of $A_n$ are approximately equal to 
\[ f\Bigl(a+i\,\frac{b-a}{n}\Bigr),\qquad i=1,\ldots,n, \]
for $n$ large enough. Similarly, if $k=2$, $n$ is a perfect square and $D=[a_1,b_1]\times [a_2,b_2]$, then, assuming we have no outliers, the eigenvalues of $A_n$ are approximately equal to
\[ f\Bigl(a_1+i_1\,\frac{b_1-a_1}{\sqrt n},\,\,a_2+i_2\,\frac{b_2-a_2}{\sqrt n}\Bigr),\qquad i_1,i_2=1,\ldots,\sqrt n, \]
for $n$ large enough. A completely analogous meaning can also be given for the singular value distribution $\{A_n\}_n\sim_\sigma f$ in \eqref{tilde_sigma}.

\paragraph{Monotone rearrangement}
The singular value symbol and the spectral symbol of a matrix-sequence are not unique.
In particular, if we have $\{A_n\}_n\sim_{\sigma,\lambda}f$ for some real measurable function $f:D\subset\mathbb R^k\to\mathbb R$ with $0<\mu_k(D)<\infty$, then we also have $\{A_n\}_n\sim_{\sigma,\lambda}f^\dag$, where
\begin{equation*}
f^\dag:(0,1)\to\mathbb R,\qquad f^\dag(t)=\inf\biggl\{u\in\mathbb R:\frac{\mu_k\{f\le u\}}{\mu_k(D)}\ge t\biggr\}.
\end{equation*}
This result follows from Definition~\ref{def:distr} and the following property \cite[Section~3.1]{sdau}:
\begin{equation*}
\int_0^1F(f^\dag(t)){\rm d}t=\frac{1}{\mu_k(D)}\int_DF(f(\xx)){\rm d}\xx,\qquad\forall\,F\in C_c(\mathbb C).
\end{equation*}
The function $f^\dag$---which is monotone non-decreasing on $(0,1)$ by definition, since the function $u\mapsto\mu_k\{f\le u\}$ is monotone non-decreasing on $\mathbb R$---is referred to as the {\em monotone rearrangement} of $f$.
A simple procedure for constructing the monotone rearrangement of a bounded and a.e.\ continuous function is reported in {\bf R1} \cite[Theorem~3.2]{sdau}.
We first recall that, given a real measurable function $f:D\subseteq\mathbb R^k\to\mathbb R$, the essential range of $f$ is defined as
\[ \mathcal{ER}(f)=\{y\in\mathbb R:\hspace{0.5pt}\mu_k\{f\in(y-\epsilon,y+\epsilon)\}>0\hspace{0.5pt}\mbox{ for every }\hspace{0.5pt}\epsilon>0\}, \]
and the essential infimum (resp., supremum) of $f$ on $D$ are defined as
\[ \mathop{\rm ess\,inf\vphantom{p}}_Df=\inf\mathcal{ER}(f),\qquad\mathop{\rm ess\,sup}_Df=\sup\mathcal{ER}(f). \]
Roughly speaking, the essential range of $f$ is the ``true'' range of $f$. For example, if
\[ f:[0,1]\to\mathbb R,\qquad f(x)=\left\{\begin{aligned}
&x, &\quad\mbox{if }x\ne\textstyle{\frac12},\\
&2, &\quad\mbox{if }x=\textstyle{\frac12},
\end{aligned}\right. \]
then
\begin{alignat*}{3}
\mathcal{ER}(f)&=[0,1],&\qquad f([0,1])&=\{2\}\cup[0,1]\setminus\{\textstyle{\frac12}\},\\
\mathop{\rm ess\,inf}_{[0,1]}f&=0=\inf_{[0,1]}f,&\qquad\mathop{\rm ess\,sup}_{[0,1]}f&=1\ne2=\sup_{[0,1]}f.
\end{alignat*}
If $f:D\subseteq\mathbb R^k\to\mathbb R$ is measurable and $\mathcal{ER}(f)$ is bounded, then $f^\dag$ is defined also at the endpoints of the interval $(0,1)$ as follows \cite[Lemma~3.1 and Definition~3.2]{sdau}:
\[ f^\dag(0)=\lim_{y\to0}f^\dag(y)=\mathop{\rm ess\,inf\vphantom{p}}_Df,\qquad f^\dag(1)=\lim_{y\to1}f^\dag(y)=\mathop{\rm ess\,sup}_Df. \]

\begin{enumerate}[leftmargin=23pt,nolistsep]
	\item[\textbf{R1.}] Let $D=[a_1,b_1]\times\cdots\times[a_k,b_k]$ be a hyperrectangle in $\mathbb R^k$ and let $f:D\to\mathbb R$ be a bounded and a.e.\ continuous function such that $\mathcal{ER}(f)=[\inf_Df,\sup_Df]$.
	For each $r\in\mathbb N=\{1,2,\ldots\}$, consider the uniform samples
	\begin{equation*}
	f\Bigl(a_1+i_1\frac{b_{1_{\vphantom{1}}}-a_{1_{\vphantom{1}}}}{r},\ldots,a_k+i_k\frac{b_{k_{\vphantom{1}}}-a_{k_{\vphantom{1}}}}{r}\Bigr),\qquad i_1,\ldots,i_k=1,\ldots,r,
	\end{equation*}
	sort them in ascending order, and put them into a vector $(s_0,s_1,\ldots,s_{r^k})$.
	Let $f^\dag_r:[0,1]\to\mathbb R$ be the piecewise linear function that interpolates the samples $(s_0,s_1,\ldots,s_{r^k})$ over the equally spaced nodes $(0,\frac{1}{r^{k\vphantom{\int}}},\frac{2}{r^{k\vphantom{\int}}},\ldots,1)$ in $[0,1]$. Then $f^\dag_r\to f^\dag$ uniformly on $[0,1]$ as $r\to\infty$.
\end{enumerate}

\paragraph{Zero-distributed sequences}
A matrix-sequence $\{Z_n\}_n$ such that $\{Z_n\}_n\sim_\sigma0$ 
is called 
zero-distributed sequence. In other words, $\{Z_n\}_n$ is zero-distributed if and only if 
\[ \lim_{n\to\infty}\frac1{n}\sum_{i=1}^nF(\sigma_i(Z_n))=F(0),\qquad\forall\,F\in C_c(\mathbb R). \]
{\bf Z1}--{\bf Z2} provide us with an important characterization of zero-distributed sequences together with a useful sufficient condition for detecting such sequences. For the proof of {\bf Z1}--{\bf Z2}, see \cite[Section~3.4]{GLT-bookI}. 
Throughout this work we use the natural convention $1/\infty=0$.
\begin{enumerate}[leftmargin=22pt,nolistsep]
	\item[\textbf{Z1.}] $\{Z_n\}_n\sim_\sigma0$ if and only if $Z_n=R_n+N_n$ with $n^{-1}{\rm rank}(R_n)\to0$ and $\|N_n\|\to0$.
	\item[\textbf{Z2.}] $\{Z_n\}_n\sim_\sigma0$ if there is a $p\in[1,\infty]$ such that $\|Z_n\|_p=o(n^{1/p})$.
\end{enumerate}

\paragraph{Sequences of diagonal sampling matrices}
Throughout this work, any finite sequence of points in $\mathbb R$ is referred to as a grid.
If $n\in\mathbb N$, $a:[0,1]\to\mathbb C$ and $\mathcal G_n=\{x_{i,n}\}_{i=1,\ldots,n}$ is a grid of $n$ points in $[0,1]$, the $n$th diagonal sampling matrix generated by $a$ corresponding to the grid $\mathcal G_n$ is the $n\times n$ diagonal matrix given by
\begin{equation*}
D_n^{\mathcal G_n}(a)=\mathop{\rm diag}_{i=1,\ldots,n}a(x_{i,n}).
\end{equation*}
Any matrix-sequence of the form $\{D_n^{\mathcal G_n}(a)\}_n$ is referred to as a sequence of diagonal sampling matrices generated by $a$. We say that the grid $\mathcal G_n=\{x_{i,n}\}_{i=1,\ldots,n}\subset[0,1]$ is asymptotically uniform (a.u.)\ in $[0,1]$ if
\[ \lim_{n\to\infty}m(\mathcal G_n)=0, \]
where
\[ m(\mathcal G_n)=\max_{i=1,\ldots,n}\biggl|x_{i,n}-\frac in\biggr| \]
is the distance of $\mathcal G_n$ from the uniform grid $\{\frac in\}_{i=1,\ldots,n}$. 
In the special case where $\mathcal G_n=\{\frac in\}_{i=1,\ldots,n}$, we write $D_n(a)$ instead of $D_n^{\mathcal G_n}(a)$.

\paragraph{Toeplitz sequences}
A Toeplitz matrix is a matrix whose entries are constant along each diagonal, i.e., a matrix of the form
\begin{equation*}
[a_{i-j}]_{i,j=1}^{n}=\begin{bmatrix}
a_0 & a_{-1} & \ \cdots & \ \ \cdots & a_{-(n-1)} \\
a_1 & \ddots & \ \ddots & \ \ & \vdots\\
\vdots & \ddots & \ \ddots & \ \ \ddots & \vdots\\
\vdots & & \ \ddots & \ \ \ddots & a_{-1}\\
a_{n-1} & \cdots & \ \cdots & \ \ a_1 & a_0
\end{bmatrix}.
\end{equation*}
If $n\in\mathbb N$ and $f:[-\pi,\pi]\to\mathbb C$ is a function in $L^1([-\pi,\pi])$, the $n$th Toeplitz matrix generated by $f$ is the $n\times n$ matrix
\[ T_n(f)=[f_{i-j}]_{i,j=1}^n, \] 
where the numbers $f_k$ are the Fourier coefficients of $f$,
\[ f_k=\frac1{2\pi}\int_{-\pi}^\pi f(\theta){\rm e}^{-{\rm i}k\theta}{\rm d}\theta,\qquad k\in\mathbb Z. \]
$\{T_n(f)\}_n$ is referred to as the Toeplitz sequence generated by $f$. 
If $f$ is a trigonometric polynomial of degree $r$, i.e., a function of the form
\begin{equation}\label{trig-sum}
f(\theta)=\sum_{k=-r}^rf_k{\rm e}^{{\rm i}k\theta},
\end{equation}
with $f_{-r},\ldots,f_r\in\mathbb C$, then it can be checked by direct computation that, as indicated by the notation, the Fourier coefficients of $f$ are given by the coefficients $f_k$ in \eqref{trig-sum} for $|k|\le r$, and are equal to $0$ for $|k|>r$. 
It follows that the Toeplitz sequence generated by $f$ is a sequence of banded matrices with bandwidth $2r+1$:
\begin{align*}
T_n(f)&=\left[\begin{array}{cc|ccccccc|cc}
f_0 & \multicolumn{1}{c}{\ \ f_{-1}} & \ \ \cdots & \ \ f_{-r} & \ \ & \ \ & \ \ & \ \ & \multicolumn{1}{c}{\ \ } & \ \ & \ \ \\
f_1 & \multicolumn{1}{c}{\ \ \ddots} & \ \ \ddots & \ \ & \ \ \ddots & \ \ & \ \ & \ \ & \multicolumn{1}{c}{\ \ } & \ \ & \ \ \\
\vdots & \multicolumn{1}{c}{\ \ \ddots} & \ \ \ddots & \ \ \ddots & \ \ & \ \ \ddots & \ \ & \ \ & \multicolumn{1}{c}{\ \ } & \ \ & \ \ \\
f_r & \multicolumn{1}{c}{\ \ } & \ \ \ddots & \ \ \ddots & \ \ \ddots & \ \ & \ \ \ddots & \ \ & \multicolumn{1}{c}{\ \ } & \ \ & \ \ \\
& \multicolumn{1}{c}{\ \ \ddots} & \ \ & \ \ \ddots & \ \ \ddots & \ \ \ddots & \ \ & \ \ \ddots & \multicolumn{1}{c}{\ \ } & \ \ & \ \ \\
\cline{3-9}
& \ \ & \ \ f_r & \ \ \cdots & \ \ f_1 & \ \ f_0 & \ \ f_{-1} & \ \ \cdots & \ \ f_{-r} & \ \ & \ \ \\
\cline{3-9}
& \multicolumn{1}{c}{\ \ } & \ \ & \ \ \ddots & \ \ & \ \ \ddots & \ \ \ddots & \ \ \ddots & \multicolumn{1}{c}{\ \ } & \ \ \ddots & \ \ \\
& \multicolumn{1}{c}{\ \ } & \ \ & \ \ & \ \ \ddots & \ \ & \ \ \ddots & \ \ \ddots & \multicolumn{1}{c}{\ \ \ddots} & \ \ & \ \ f_{-r}\\
& \multicolumn{1}{c}{\ \ } & \ \ & \ \ & \ \ & \ \ \ddots & \ \ & \ \ \ddots & \multicolumn{1}{c}{\ \ \ddots} & \ \ \ddots & \ \ \vdots\\
& \multicolumn{1}{c}{\ \ } & \ \ & \ \ & \ \ & \ \ & \ \ \ddots & \ \ & \multicolumn{1}{c}{\ \ \ddots} & \ \ \ddots & \ \ f_{-1}\\
& \multicolumn{1}{c}{\ \ } & \ \ & \ \ & \ \ & \ \ & \ \ & \ \ f_r & \multicolumn{1}{c}{\ \ \cdots} & \ \ f_1 & \ \ f_0
\end{array}\right].
\end{align*}

\paragraph{Approximating classes of sequences}
The notion of approximating classes of sequences (a.c.s.)\ is the fundamental concept on which the theory of GLT sequences is based.
We use the abbreviation ``a.c.s.'' for both the singular ``approximating class of sequences'' and the plural ``approximating classes of sequences''.

\begin{definition}[\textbf{approximating class of sequences}]\label{S.a.c.s.}
Let $\{A_n\}_n$ be a matrix-sequence and let $\{\{B_{n,m}\}_n\}_m$ be a sequence of matrix-sequences. We say that $\{\{B_{n,m}\}_n\}_m$ is an approximating class of sequences (a.c.s.)\ for $\{A_n\}_n$ if the following condition is met: for every $m$ there exists $n_m$ such that, for $n\ge n_m$,
\begin{equation*}
A_n=B_{n,m}+R_{n,m}+N_{n,m},\qquad{\rm rank}(R_{n,m})\le c(m)n,\qquad\|N_{n,m}\|\le\omega(m),
\end{equation*}
where $n_m,\,c(m),\,\omega(m)$ depend only on $m$, and $\displaystyle\lim_{m\to\infty}c(m)=\lim_{m\to\infty}\omega(m)=0$.
\end{definition}

Roughly speaking, $\{\{B_{n,m}\}_n\}_m$ is an a.c.s.\ for $\{A_n\}_n$ if, for every sufficiently large $m$, the sequence $\{B_{n,m}\}_n$ approximates $\{A_n\}_n$ in the sense that $A_n$ is eventually equal to $B_{n,m}$ plus a small-rank matrix (with respect to the matrix size $n$) plus a small-norm matrix. 

The notion of a.c.s.\ is a notion of convergence in the space of matrix-sequences
\[ \mathscr E=\{\{A_n\}_n:\{A_n\}_n\mbox{ is a matrix-sequence}\}, \]
because there exists a distance $d_{\rm a.c.s.}$ on $\mathscr E$ such that
\[ \{\{B_{n,m}\}_n\}_m\textup{ is an a.c.s.\ for }\{A_n\}_n\ \iff\ d_{\rm a.c.s.}(\{\{B_{n,m}\}_n\}_m,\{A_n\}_n)\to0\textup{ as }m\to\infty; \]
see \cite[Section~5.2]{GLT-bookI} and \cite{tau_acs}.
The theory of a.c.s.\ may then be interpreted as an approximation theory for matrix-sequences, and for this reason we will use the convergence notation $\{B_{n,m}\}_n\stackrel{\rm a.c.s.}{\longrightarrow}\{A_n\}_n$ to indicate that $\{\{B_{n,m}\}_n\}_m$ is an a.c.s.\ for $\{A_n\}_n$. For the proof of the next property, see \cite[Corollary~5.3]{GLT-bookI}.
\begin{enumerate}[leftmargin=38pt,nolistsep]
	\item[\textbf{ACS1.}] Let $p\in[1,\infty]$ and suppose that for every $m$ there exists $n_m$ such that, for $n\ge n_m$, 
	\[ \|A_n-B_{n,m}\|_p\le\epsilon(m,n)n^{1/p}, \]
	where $\displaystyle\lim_{m\to\infty}\limsup_{n\to\infty}\epsilon(m,n)=0$. Then $\{B_{n,m}\}_n\stackrel{\rm a.c.s.}{\longrightarrow}\{A_n\}_n$.
\end{enumerate}

\paragraph{Generalized locally Toeplitz sequences}
A generalized locally Toeplitz (GLT) sequence $\{A_n\}_n$ is a special matrix-sequence equipped with a measurable function $\kappa:[0,1]\times[-\pi,\pi]\to\mathbb C$, the so-called {\em symbol}. We use the notation $\{A_n\}_n\sim_{\rm GLT}\kappa$ to indicate that $\{A_n\}_n$ is a GLT sequence with symbol~$\kappa$.

The main properties of GLT sequences are summarized in the following list. The corresponding proofs can be found in \cite[Chapter~8]{GLT-bookI}, except for {\bf GLT2}, whose proof can be found in \cite{NLAA}, and the second item of {\bf GLT3}, whose proof can be found in \cite{a.u.}. After the list of properties {\bf GLT1}--{\bf GLT7}, we give a formal definition of GLT sequences that can be inferred from \cite[Theorem~8.6]{GLT-bookI}, and we explain the origin of the name ``generalized locally Toeplitz sequence''.

Throughout this work, if $A$ is a matrix, we denote by $A^\dag$ the Moore--Penrose pseudoinverse of $A$. What is relevant for our purposes is that $A^\dag=A^{-1}$ whenever $A$ is invertible. For more details on the pseudoinverse of a matrix, see \cite{Bini,GV}. If $A$ is a diagonalizable matrix and $f$ is a function defined at each point of $\Lambda(A)$, we denote by $f(A)$ the unique matrix such that $f(A)\mathbf v=f(\lambda)\mathbf v$ whenever $A\mathbf v=\lambda\mathbf v$. Note that a spectral decomposition of $A$ immediately implies a spectral decomposition of $f(A)$:
\begin{equation}\label{A->f(A)}
A=V\begin{bmatrix}\lambda_1 & & \\ & \ddots & \\ & & \lambda_n\end{bmatrix}V^{-1}\ \implies\ f(A)=V\begin{bmatrix}f(\lambda_1) & & \\ & \ddots & \\ & & f(\lambda_n)\end{bmatrix}V^{-1}.
\end{equation}
For more on matrix functions, we refer the reader to Higham's book~\cite{Higham}.
Hereafter, the composite function $f\circ g$ is denoted by $f(g)$. We recall that, if $f_m,f:D\subseteq\mathbb R^k\to\mathbb C$ are measurable functions, then $f_m\to f$ in measure if, by definition,
\[ \lim_{m\to\infty}\mu_k\{|f_m-f|>\epsilon\}=0,\qquad\forall\,\epsilon>0. \]
\begin{enumerate}[leftmargin=37.5pt,nolistsep]
	\item[\textbf{GLT1.}] If $\{A_n\}_n\sim_{\rm GLT}\kappa$ then $\{A_n\}_n\sim_\sigma\kappa$. If $\{A_n\}_n\sim_{\rm GLT}\kappa$ and each $A_n$ is Hermitian then $\{A_n\}_n\sim_\lambda\kappa$.
	\item[\textbf{GLT2.}] If $\{A_n\}_n\sim_{\rm GLT}\kappa$ and $A_n=X_n+Y_n$, where
	\begin{itemize}[nolistsep,leftmargin=*]
		\item every $X_n$ is Hermitian,
		\item $\|Y_n\|_2=o(n^{1/2})$,
	\end{itemize}
	then $\{A_n\}_n\sim_\lambda\kappa$.
	\item[\textbf{GLT3.}] We have
	\begin{itemize}[nolistsep,leftmargin=*]
		\item $\{T_n(f)\}_n\sim_{\rm GLT}\kappa(x,\theta)=f(\theta)$ if $f\in L^1([-\pi,\pi])$,
		\item $\{D_n^{\mathcal G_n}(a)\}_n\sim_{\rm GLT}\kappa(x,\theta)=a(x)$ if $a:[0,1]\to\mathbb C$ is continuous a.e.\ and $\mathcal G_n$ is a.u.\ in $[0,1]$, 
		\item $\{Z_n\}_n\sim_{\rm GLT}\kappa(x,\theta)=0$ if and only if $\{Z_n\}_n\sim_\sigma0$.
	\end{itemize}
	\item[\textbf{GLT4.}] If $\{A_n\}_n\sim_{\rm GLT}\kappa$ and $\{B_n\}_n\sim_{\rm GLT}\xi$ then
	\begin{itemize}[nolistsep,leftmargin=*]
		\item $\{A_n^*\}_n\sim_{\rm GLT}\overline\kappa$,
		\item $\{\alpha A_n+\beta B_n\}_n\sim_{\rm GLT}\alpha\kappa+\beta\xi$ for all $\alpha,\beta\in\mathbb C$,
		\item $\{A_nB_n\}_n\sim_{\rm GLT}\kappa\xi$.
	\end{itemize}
	\item[\textbf{GLT5.}] If $\{A_n\}_n\sim_{\rm GLT}\kappa$ and $\kappa\ne0$ a.e.\ then $\{A_n^\dag\}_n\sim_{\rm GLT}\kappa^{-1}$.
	\item[\textbf{GLT6.}] If $\{A_n\}_n\sim_{\rm GLT}\kappa$ and each $A_n$ is Hermitian then $\{f(A_n)\}_n\sim_{\rm GLT}f(\kappa)$ for every continuous function $f:\mathbb C\to\mathbb C$.
	\item[\textbf{GLT7.}] $\{A_n\}_n\sim_{\rm GLT}\kappa$ if and only if there exist GLT sequences $\{B_{n,m}\}_n\sim_{\rm GLT}\kappa_m$ such that $\{B_{n,m}\}_n\stackrel{\rm a.c.s.}{\longrightarrow}\{A_n\}_n$ and $\kappa_m\to\kappa$ in measure.
\end{enumerate}
{\bf GLT1}--{\bf GLT2} are the main distribution results for GLT sequences. Roughly speaking, they say that the symbol $\kappa$ of a GLT sequence $\{A_n\}_n$ is automatically a singular value symbol for $\{A_n\}_n$, and it is also a spectral symbol for $\{A_n\}_n$ if the matrices $A_n$ are Hermitian or ``almost Hermitian''.
{\bf  GLT3} lists the fundamental examples of GLT sequences, from which one can construct many other GLT sequences via {\bf GLT4}. {\bf GLT4} says that the set of GLT sequences is a *-algebra. In practice, this means that, if $\{A_{n,i}\}_n$ is a GLT sequence with symbol $\kappa_i$ for $i=1,\ldots,t$, and if $A_n={\rm ops}(A_{n,1},\ldots,A_{n,t})$ is obtained from $A_{n,1},\ldots,A_{n,t}$ by means of certain algebraic operations ``ops'' such as conjugate transpositions, linear combinations and products, then $\{A_n\}_n$ is a GLT sequence with symbol $\kappa={\rm ops}(\kappa_1,\ldots,\kappa_t)$ obtained by performing the same algebraic operations on the symbols $\kappa_1,\ldots,\kappa_t$. {\bf GLT5}--{\bf GLT6} show that the set of GLT sequences is closed under operations more general than sums and products, such as pseudoinversions and matrix functions. {\bf GLT7} is a closure property of GLT sequences that can be rephrased as follows: if a sequence of GLT sequences $\{B_{n,m}\}_n$ converges to a matrix-sequence $\{A_n\}_n$ (in the a.c.s.\ sense), and if the corresponding sequence of symbols $\kappa_m$ converges to a function $\kappa$ (in measure), then $\{A_n\}_n$ is a GLT sequence with symbol $\kappa$.


\begin{definition}[\textbf{generalized locally Toeplitz sequence}]\label{def:GLT}
Let $\{A_n\}_n$ be a matrix-sequence and let $\kappa:[0,1]\times[-\pi,\pi]\to\mathbb C$ be a measurable function. We say that $\{A_n\}_n$ is a generalized locally Toeplitz (GLT) sequence with symbol $\kappa$, and we write $\{A_n\}_n\sim_{\rm GLT}\kappa$, if there exist functions $a_{i,m}$, $f_{i,m}$, $i=1,\ldots,N_m$, such that:
\begin{itemize}[nolistsep,leftmargin=*]
	\item $a_{i,m}:[0,1]\to\mathbb C$ is continuous a.e.\ and $f_{i,m}\in L^1([-\pi,\pi])$. \vspace{3pt}
	\item $\kappa_m(x,\theta)=\sum_{i=1}^{N_m}a_{i,m}(x)f_{i,m}(\theta)\to\kappa(x,\theta)$ in measure on $[0,1]\times[-\pi,\pi]$. \vspace{3pt}
	\item $\{A_{n,m}\}_n=\bigl\{\sum_{i=1}^{N_m}D_n(a_{i,m})T_n(f_{i,m})\bigr\}_n\stackrel{\rm a.c.s.}\longrightarrow\{A_n\}_n$. 
	\end{itemize}
\end{definition}

\begin{remark}[\textbf{origin of the name ``generalized locally Toeplitz sequence''}]
The prototype of a ``locally Toeplitz'' sequence is the matrix-sequence $\{A_n\}_n$, where 
\begin{align*}
A_n&=D_n(a)T_n(f)\\
&=\begin{bmatrix}
a(x_1)f_0 & a(x_1)f_{-1} & \ a(x_1)f_{-2} & \ \cdots & \ \ \cdots & a(x_1)f_{-(n-1)} \\[10pt]
a(x_2)f_1 & \ddots & \ \ddots & \ \ddots & \ \ & \vdots\\[10pt]
a(x_3)f_2 & \ddots & \ \ddots & \ \ddots & \ \ \ddots & \vdots\\[10pt]
\vdots & \ddots & \ \ddots & \ \ddots & \ \ \ddots & a(x_{n-2})f_{-2}\\[10pt]
\vdots & & \ \ddots & \ \ddots & \ \ \ddots & a(x_{n-1})f_{-1}\\[10pt]
a(x_n)f_{n-1} & \cdots & \ \cdots & \ a(x_n)f_2 & \ \ a(x_n)f_1 & a(x_n)f_0
\end{bmatrix},
\end{align*}
$x_i=\frac in$ for $i=1,\ldots,n$, $f\in L^1([-\pi,\pi])$, and $a:[0,1]\to\mathbb C$ is continuous.
The name is due to the fact that the entries of $A_n$ along any diagonal, although not constant as in the case of a Toeplitz matrix, vary ``gradually'' from the top to the bottom of the diagonal. For example, the entries of $A_n$ along the main diagonal are given by $a(x_i)f_0$, $i=1,\ldots,n$. The distance between two consecutive entries is bounded by $\omega_a(\frac1n)|f_0|$, where $\omega_a(\cdot)$ is the modulus of continuity of $a$ and $\omega_a(\delta)\to0$ as $\delta\to0$; see Section~\ref{moduli}.
The transition from the first entry $a(x_1)f_0\approx a(0)f_0$ to the last entry $a(1)f_0$ is more and more gradual as $n$ increases and, in a sense, we can say that the transition is continuous in the limit as $n\to\infty$, just as the function $a(x)f_0$.
As a consequence, any submatrix of $A_n$ made up of $k_n=o(n)$ consecutive rows and columns possesses a sort of approximate Toeplitz structure. 
For instance, the $2\times 2$ leading principal submatrix
\[ \begin{bmatrix}a(x_1)f_0 & a(x_1)f_{-1}\\[3pt]
a(x_2)f_1 & a(x_2)f_0
\end{bmatrix} \]
is approximately equal to
\[ a(x_1)\begin{bmatrix}f_0 & f_{-1}\\ f_1 & f_0 \end{bmatrix} = a(x_1)T_2(f), \]
because the difference between these two matrices goes to $0$ in spectral norm as $n\to\infty$.
Similarly, if $B_{k_n}$ is a principal submatrix of size $k_n$ made up of $k_n=o(n)$ consecutive rows and columns of $A_n$, then $B_{k_n}\approx a(x_i)T_{k_n}(f)$, where $a(x_i)$ is any of the evaluations of $a(x)$ appearing in $B_{k_n}$. More precisely, one can prove that 
\begin{equation*} 
B_{k_n}=a(x_i)T_{k_n}(f)+E_{k_n}, 
\end{equation*}
where the error $E_{k_n}$ tends to the zero matrix in spectral norm as $n\to\infty$, as a consequence of the fact that 
$\omega_a(k_n/n)\to0$. In other words, if we explore ``locally'' the matrix $A_n$, using an ideal microscope and considering a large value of $n$, then we realize that the ``local'' structure of $A_n$ is approximately the Toeplitz structure generated by $f$ and weighted through the function $a(x)$.
This original intuition led Tilli to coin the name ``locally Toeplitz sequences'' \cite{Tilli98}. The name was soon changed to ``generalized locally Toeplitz sequences'' by Serra-Capizzano, who generalized Tilli's intuition and introduced for the first time the modern notion of GLT sequences \cite{glt_1,glt_2}.
\end{remark}

\section{Applications}\label{Ch10}

In this section we present applications of the theory of GLT sequences to the spectral analysis of DE discretization matrices.
In practice, we show how to compute the singular value and spectral distribution of matrix-sequences arising from a DE discretization through the ``GLT tools'' presented in Section~\ref{Ch9}. 
We first consider FD discretizations in Section~\ref{sec:FDs} and then we move to FE discretizations in Section~\ref{sec:FEs}.
For further applications not included herein, we refer the reader to \cite[Chapter~10]{GLT-bookI}.
Before starting, we collect in Section~\ref{preli} some auxiliary results.

\subsection{Preliminaries}\label{preli}

\subsubsection{Matrix-Norm Inequalities}
If $1\le p\le\infty$, we denote by $|\cdot|_p$ both the $p$-norm of vectors and the associated induced norm for matrices:
\begin{align*}
|\xx|_p&=\left\{\begin{aligned}&\textstyle{\left(\sum_{i=1}^m|x_i|^p\right)^{1/p}}, &&\quad\mbox{if $1\le p<\infty$,}\\[3pt]
&\textstyle{\max_{i=1,\ldots,m}|x_i|}, &&\quad\mbox{if $p=\infty$,}\end{aligned}\right.\qquad \xx\in\mathbb C^m,\\[5pt]
|X|_p&=\max_{\substack{\xx\in\mathbb C^m\\\xx\ne\mathbf{0}}}\frac{|X\xx|_p}{|\xx|_p}, \qquad X\in\mathbb C^{m\times m}.
\end{align*}
The 2-norm $|\cdot|_2$ is also known as the spectral (or Euclidean) norm and it is preferably denoted by $\|\cdot\|$.
An important inequality involving the $p$-norms with $p=1,2,\infty$ is the following \cite[p.~29]{GLT-bookI}:
\begin{align}
\|X\|&\le\sqrt{|X|_1|X|_\infty},\qquad X\in\mathbb C^{m\times m}.\label{2-norm}
\end{align}
Since $|X|_1=\max_{j=1,\ldots,m}\sum_{i=1}^m|X_{ij}|$ and $|X|_\infty=\max_{i=1,\ldots,m}\sum_{j=1}^m|X_{ij}|$, the inequality \eqref{2-norm} is particularly useful to estimate the spectral norm of a matrix when we have bounds for its components.

As mentioned in Section~\ref{Ch9}, the Schatten $p$-norm of an $m\times m$ matrix $X$ is defined as the $p$-norm of the vector $(\sigma_1(X),\ldots,\sigma_m(X))$ formed by the singular values of $X$.
The Schatten $p$-norms are deeply studied in Bhatia's book \cite{Bhatia}. 
Here, we just recall a basic trace-norm inequality \cite[Section~2.4.3]{GLT-bookI}:
\begin{align}
\|X\|_1&\le\sum_{i,j=1}^m|X_{ij}|,\qquad X\in\mathbb C^{m\times m}.\label{trace-norm-Lusin}
\end{align}

\subsubsection{Modulus of Continuity and Integral Continuity}\label{moduli}

\paragraph{Modulus of continuity}
If $f:D\subseteq\mathbb C^k\to\mathbb C$, the modulus of continuity of $f$ is defined as
\[ \displaystyle\omega_f(\delta)=\sup_{\substack{\xx,\yy\in D\\\|\xx-\yy\|\le\delta}}|f(\xx)-f(\yy)|,\qquad\delta>0. \]
Note that the function $\delta\mapsto\omega_f(\delta)$ is monotone non-decreasing on $(0,\infty)$.
What is relevant for our purposes is that $\omega_f(\delta)\to0$ as $\delta\to0$ for every uniformly continuous function $f$ on $D$ (this is actually the definition of a uniformly continuous function). In particular, $\omega_f(\delta)\to0$ as $\delta\to0$ whenever $f$ is continuous on $D$ and $D$ is compact, because any continuous function on a compact set is uniformly continuous by the Heine--Cantor theorem \cite[Theorem~4.19]{Rudinino}.

\paragraph{Modulus of integral continuity}
If $f:D\subseteq\mathbb R^k\to\mathbb C$ is a function in $L^1(D)$, we define the modulus of integral continuity of $f$ as
\[ \displaystyle\omega_f^{\rm int}(\delta)=\sup_{\substack{E\subseteq D\textup{ measurable}\\\mu_k(E)\le\delta}}\int_E|f(\xx)|{\rm d}\xx,\qquad\delta>0. \]
Note that the function $\delta\mapsto\omega_f^{\rm int}(\delta)$ is monotone non-decreasing on $(0,\infty)$.
What is relevant for our purposes is that $\omega_f^{\rm int}(\delta)\to0$ as $\delta\to0$ for every $f\in L^1(D)$. This property is known as the absolute continuity of the Lebesgue integral and is proved below for the reader's convenience. Throughout this work, we denote by $\chi_E$ the characteristic (indicator) function of the set $E$.

\begin{theorem}[\textbf{absolute continuity of the Lebesgue integral}]\label{ac}
Let $f:D\subseteq\mathbb R^k\to\mathbb C$ be a function in $L^1(D)$. Then, for every $\epsilon>0$ there exists $\delta=\delta(\epsilon)>0$ such that
\[ \int_E|f(\xx)|{\rm d}\xx\le\epsilon \]
for every measurable $E\subseteq D$ with $\mu_k(E)\le\delta$.
\end{theorem}
\begin{proof}
We first note that, by the dominated convergence theorem \cite[Theorem~1.34]{Rudinone},
\[ \lim_{M\to\infty}\int_{\{|f|>M\}}|f(\xx)|{\rm d}\xx=\lim_{M\to\infty}\int_D|f(\xx)|\chi_{\{|f|>M\}}(\xx){\rm d}\xx=0, \]
because $|f(\xx)|\chi_{\{|f|>M\}}(\xx)\to0$ pointwise as $M\to\infty$ and $|f(\xx)|\chi_{\{|f|>M\}}(\xx)\le|f(\xx)|$ for every $\xx\in D$ and every~$M$.
Now, fix $\epsilon>0$. Choose $M=M(\epsilon)>0$ such that
\[ \int_{\{|f|>M\}}|f(\xx)|{\rm d}\xx\le\frac\epsilon2, \]
and take $\delta=\delta(\epsilon)=\dfrac\epsilon{2M}$. Then, for every measurable $E\subseteq D$ such that $\mu_k(E)\le\delta$, we have
\begin{align*}
\int_E|f(\xx)|{\rm d}\xx&=\int_{E\cap\{|f|>M\}}|f(\xx)|{\rm d}\xx+\int_{E\cap\{|f|\le M\}}|f(\xx)|{\rm d}\xx\\
&\le\int_{\{|f|>M\}}|f(\xx)|{\rm d}\xx+\int_{E\cap\{|f|\le M\}}M\hspace{0.5pt}{\rm d}\xx\le\frac\epsilon2+M\delta=\epsilon.\qedhere
\end{align*}
\end{proof}

As a consequence of Theorem~\ref{ac}, given a function $f:D\subseteq\mathbb R^k\to\mathbb C$ belonging to $L^1(D)$, for every $\epsilon>0$ there exists $\hat\delta=\hat\delta(\epsilon)>0$ such that $\omega_f^{\rm int}(\hat\delta)\le\epsilon$. This implies that $\omega_f^{\rm int}(\delta)\le\epsilon$ for every $\delta\le\hat\delta$, because the function $\delta\mapsto\omega_f^{\rm int}(\delta)$ is monotone non-decreasing. We conclude that $\lim_{\delta\to0}\omega_f^{\rm int}(\delta)=0$, as claimed above.

\begin{remark}
The modulus of integral continuity introduced in this section should not be confused with the more famous integral modulus of continuity~\cite{Rees}.
While the latter is useless for our purposes, the former will be used in the context of FE discretizations to obtain spectral distribution results under minimal integrability assumptions on the DE coefficients; see Section~\ref{sec:FEs} (in particular, Theorem~\ref{FE_T1}).
\end{remark}

\subsubsection{GLT Preconditioning}

The next theorem is an important result in the context of GLT preconditioning, but will be used only in Section~\ref{FE_eigp}.
The reader may decide to skip it on first reading and come back here afterwards, just before moving on to Section~\ref{FE_eigp}.

\begin{theorem}\label{exe-preconditioning}
Let $\{A_n\}_n$ be a sequence of Hermitian matrices such that $\{A_n\}_n\sim_{\rm GLT}\kappa$, and let $\{P_n\}_n$ be a sequence of Hermitian positive definite (HPD) matrices such that $\{P_n\}_n\sim_{\rm GLT}\xi$ with $\xi\ne0$ a.e. Then, the sequence of preconditioned matrices $P_n^{-1}A_n$ satisfies
\begin{equation}\label{GLTp-GLT}
\{P_n^{-1}A_n\}_n\sim_{\rm GLT}\xi^{-1}\kappa,
\end{equation}
and
\begin{equation}\label{GLTp-sigla}
\{P_n^{-1}A_n\}_n\sim_{\sigma,\lambda}\xi^{-1}\kappa.
\end{equation}
\end{theorem}
\begin{proof}
The GLT relation \eqref{GLTp-GLT} 
is a direct consequence of {\bf GLT4}--{\bf GLT5}. The singular value distribution in \eqref{GLTp-sigla} 
follows immediately from \eqref{GLTp-GLT} and {\bf GLT1}. The only difficult part is to prove the spectral distribution in \eqref{GLTp-sigla}. 
Since $P_n$ is HPD, the eigenvalues of $P_n$ are positive and the matrices $P_n^{1/2}$ and $P_n^{-1/2}$ are well-defined as functions of $P_n$. Indeed, we have $P_n^{1/2}=\varphi(P_n)$ with $\varphi(\lambda)=\lambda^{1/2}$ and $P_n^{-1/2}=\psi(P_n)$ with $\psi(\lambda)=\lambda^{-1/2}$, where $\varphi$ and $\psi$ are well-defined on $\Lambda(P_n)\subset(0,\infty)$. Note that $P_n^{1/2}P_n^{-1/2}=I_n$ and both $P_n^{1/2}$ and $P_n^{-1/2}$ are HPD.\,\footnote{\,These properties follow from \eqref{A->f(A)} applied with $A=P_n$ and with a unitary matrix $V$, which is possible because $P_n$ is HPD.}
Moreover, we have
\begin{equation}\label{sim.sim}
P_n^{-1}A_n\sim P_n^{1/2}(P_n^{-1}A_n)P_n^{-1/2}=P_n^{-1/2}A_nP_n^{-1/2},
\end{equation}
where $X\sim Y$ means that $X$ is similar to $Y$. The good news is that $P_n^{-1/2}A_nP_n^{-1/2}$ is Hermitian, because $P_n^{-1/2}$ and $A_n$ are both Hermitian. By {\bf GLT4}--{\bf GLT5} and {\bf GLT6} applied with $f(\lambda)=|\lambda|^{1/2}$, we have
\begin{align}
\{P_n^{1/2}\}_n=\{f(P_n)\}_n&\sim_{\rm GLT}f(\xi)=|\xi|^{1/2},\notag\\
\{P_n^{-1/2}\}_n=\{(P_n^{1/2})^{-1}\}_n&\sim_{\rm GLT}(|\xi|^{1/2})^{-1}=|\xi|^{-1/2},\notag\\
\{P_n^{-1/2}A_nP_n^{-1/2}\}_n&\sim_{\rm GLT}|\xi|^{-1/2}\kappa|\xi|^{-1/2}=|\xi|^{-1}\kappa.\label{illabel}
\end{align}
Note that $\xi\ge0$ a.e.\ by {\bf S1}, since $\{P_n\}_n\sim_\lambda\xi$ by {\bf GLT1} and $\Lambda(P_n)\subset(0,\infty)$ because $P_n$ is HPD.
Hence, $|\xi|^{-1}\kappa=\xi^{-1}\kappa$ a.e., and we infer from \eqref{illabel} and the definition of GLT sequences (Definition~\ref{def:GLT}) that
\[ \{P_n^{-1/2}A_nP_n^{-1/2}\}_n\sim_{\rm GLT}\xi^{-1}\kappa. \]
Since $P_n^{-1/2}A_nP_n^{-1/2}$ is Hermitian, {\bf GLT1} yields
\[ \{P_n^{-1/2}A_nP_n^{-1/2}\}_n\sim_\lambda\xi^{-1}\kappa. \]
Thus, by the similarity \eqref{sim.sim}, $\{P_n^{-1}A_n\}_n\sim_\lambda\xi^{-1}\kappa$.
\end{proof}

\subsubsection{Arrow-Shaped Sampling Matrices}\label{assm}
If $n\in\mathbb N$, $a:[0,1]\to\mathbb C$ and $\mathcal G_n=\{x_{i,n}\}_{i=1,\ldots,n}$ is a grid of $n$ points in $[0,1]$, the $n$th arrow-shaped sampling matrix generated by $a$ corresponding to the grid $\mathcal G_n$ is the $n\times n$ symmetric matrix given by 
\begin{equation}\label{a-s}
(S_n^{\mathcal G_n}(a))_{ij}=(D_n^{\mathcal G_n}(a))_{\min(i,j),\min(i,j)}=a(x_{\min(i,j),n}),\qquad i,j=1,\ldots,n,
\end{equation}
i.e.,
\begin{equation*} 
S_n^{\mathcal G_n}(a)=\begin{bmatrix}
a(x_{1,n}) & a(x_{1,n}) & a(x_{1,n}) & \cdots & \ \cdots & a(x_{1,n})\\[5pt]
a(x_{1,n}) & a(x_{2,n}) & a(x_{2,n}) & \cdots & \ \cdots & a(x_{2,n})\\[5pt]
a(x_{1,n}) & a(x_{2,n}) & a(x_{3,n}) & \cdots & \ \cdots & a(x_{3,n})\\
\vdots & \vdots & \vdots & \ddots & & \vdots\\
\vdots & \vdots & \vdots & & \ \ddots & \vdots\\
a(x_{1,n}) & a(x_{2,n}) & a(x_{3,n}) & \cdots & \ \cdots & a(x_{n,n})
\end{bmatrix}.
\end{equation*}
The name is due to the fact that, if we imagine to color the matrix $S_n^{\mathcal G_n}(a)$ by assigning the color $i$ to the entries $a(x_{i,n})$, the resulting picture looks like a sort of arrow pointing toward the upper left corner. In the special case where $\mathcal G_n=\{\frac in\}_{i=1,\ldots,n}$, we write $S_n(a)$ instead of $S_n^{\mathcal G_n}(a)$.
Throughout this work, if $X,Y\in\mathbb C^{n\times n}$, we denote by $X\circ Y$ the componentwise (Hadamard) product of $X$ and $Y$: 
\[ (X\circ Y)_{ij}=X_{ij}Y_{ij},\qquad i,j=1,\ldots,n. \]
Moreover, if $f:D\to\mathbb C$ is any function, we set $\|f\|_\infty=\sup_{\xi\in D}|f(\xi)|$. 
Note that $\|f\|_\infty<\infty$ if and only if $f$ is bounded over its domain $D$.

\begin{theorem}\label{Dtilde_lemma}
Let $a:[0,1]\to\mathbb C$ be continuous, let $f$ be a trigonometric polynomial of degree $\le r$, and let $\mathcal G_n=\{x_{i,n}\}_{i=1,\ldots,n}$ be an a.u.\ grid in $[0,1]$. Then, for every $n\in\mathbb N$, we have
\begin{equation}\label{DtD}
\|S_n^{\mathcal G_n}(a)\circ T_n(f)-D_n^{\mathcal G_n}(a)T_n(f)\|_2\le r^{1/2}\|f\|_\infty n^{1/2}\omega_a\Bigl(\frac{r}{n}+2m(\mathcal G_n)\Bigr)=o(n^{1/2}),
\end{equation}
and
\begin{equation}\label{annex2}
\{S_n^{\mathcal G_n}(a)\circ T_n(f)\}_n\sim_{\rm GLT}a(x)f(\theta).
\end{equation}
\end{theorem}
\begin{proof}
It suffices to prove \eqref{DtD}.
Indeed, \eqref{DtD} implies that
\[ \{S_n^{\mathcal G_n}(a)\circ T_n(f)-D_n^{\mathcal G_n}(a)T_n(f)\}_n\sim_\sigma0 \]
by {\bf Z2} and 
\[ \{S_n^{\mathcal G_n}(a)\circ T_n(f)-D_n^{\mathcal G_n}(a)T_n(f)\}_n\sim_{\rm GLT}0 \]
by {\bf GLT3}. Hence, the GLT relation \eqref{annex2} follows from {\bf GLT3}--{\bf GLT4} and the decomposition
\[ S_n^{\mathcal G_n}(a)\circ T_n(f)=D_n^{\mathcal G_n}(a)T_n(f)+(S_n^{\mathcal G_n}(a)\circ T_n(f)-D_n^{\mathcal G_n}(a)T_n(f)). \]
Let us prove \eqref{DtD}. A direct comparison between
\[ S_n^{\mathcal G_n}(a)\circ T_n(f)\hspace{-1pt}=\hspace{-3pt}\begin{bmatrix}
a_1f_0 & a_1f_{-1} & \cdots & a_1f_{-r} & & & & & \\
a_1f_1 & \ddots & \ddots & & \ddots & & & & \\
\vdots & \ddots & \ddots & \ddots & & \ddots & & & \\
a_1f_r & & \ddots & \ddots & \ddots & & \ddots & & \\
& \ddots & & \ddots & \ddots & \ddots & & \ddots & \\
& & \ddots & & \ddots & \ddots & \ddots & & a_{n-r}f_r\\ 
& & & \ddots & & \ddots & \ddots & \ddots & \vdots\\
& & & & \ddots & & \ddots & \ddots & a_{n-1}f_{-1}\\
& & & & & a_{n-r}f_r & \cdots & a_{n-1}f_1 & a_nf_0
\end{bmatrix} \]
and 
\[ D_n^{\mathcal G_n}(a)T_n(f)\hspace{-1pt}=\hspace{-3pt}\begin{bmatrix}
a_1f_0 & a_1f_{-1} & \cdots & a_1f_{-r} & & & & & \\
a_2f_1 & \ddots & \ddots & & \ddots & & & & \\
\vdots & \ddots & \ddots & \ddots & & \ddots & & & \\
a_{r+1}f_r & & \ddots & \ddots & \ddots & & \ddots & & \\
& \ddots & & \ddots & \ddots & \ddots & & \ddots & \\
& & \ddots & & \ddots & \ddots & \ddots & & a_{n-r}f_r\\ 
& & & \ddots & & \ddots & \ddots & \ddots & \vdots\\
& & & & \ddots & & \ddots & \ddots & a_{n-1}f_{-1}\\
& & & & & a_nf_r & \cdots & a_nf_1 & a_nf_0
\end{bmatrix}, \]
where $a_i=a(x_{i,n})$ for all $i=1,\ldots,n$, shows that these two matrices have the same main diagonal and the same upper triangular part, meaning that the only nonzero entries of their difference are located in the (strictly) lower triangular part. Moreover, only the first $r$ diagonals of this lower triangular part can be nonzero, because $f$ is a trigonometric polynomial of degree $\le r$ and so its Fourier coefficients $f_k$ are zero for $k>r$. Thus,
\begin{align*}
\|S_n^{\mathcal G_n}(a)\circ T_n(f)-D_n^{\mathcal G_n}(a)T_n(f)\|_2^2&=\sum_{k=1}^r\sum_{\ell=1}^{n-k}\bigl|(S_n^{\mathcal G_n}(a)\circ T_n(f))_{k+\ell,\ell}-(D_n^{\mathcal G_n}(a)T_n(f))_{k+\ell,\ell}\bigr|^2\\
&=\sum_{k=1}^r\sum_{\ell=1}^{n-k}\bigl|a_\ell f_k-a_{k+\ell}f_k\bigr|^2=\sum_{k=1}^r|f_k|^2\sum_{\ell=1}^{n-k}\bigl|a(x_{\ell,n})-a(x_{k+\ell,n})\bigr|^2\\
&\le r\|f\|_\infty^2n\,\omega_a\Bigl(\frac rn+2m(\mathcal G_n)\Bigr)^2,
\end{align*}
where the last inequality follows from two observations:
\begin{itemize}[nolistsep,leftmargin=*]
	\item for every $k\in\mathbb Z$, the Fourier coefficients of $f$ satisfy
	\[ |f_k|=\left|\frac1{2\pi}\int_{-\pi}^\pi f(\theta){\rm e}^{-{\rm i}k\theta}{\rm d}\theta\right|\le\frac1{2\pi}\int_{-\pi}^\pi|f(\theta)|{\rm d}\theta\le\frac1{2\pi}\int_{-\pi}^\pi\|f\|_{\infty}{\rm d}\theta=\|f\|_\infty; \]
	\item for every $k=1,\ldots,r$ and every $\ell=1,\ldots,n-k$, we have
	\[ |x_{\ell,n}-x_{k+\ell,n}|\le\left|x_{\ell,n}-\frac\ell n\right|+\left|\frac\ell n-\frac{k+\ell}n\right|+\left|\frac{k+\ell}n-x_{k+\ell,n}\right|\le m(\mathcal G_n)+\frac rn+m(\mathcal G_n). \tag*{\qedhere} \]
\end{itemize}
\end{proof}

\subsection{FD Discretization of DEs}\label{sec:FDs}

\subsubsection{FD Discretization of Diffusion Equations}\label{1D-case}

Consider the following second-order differential problem:
\begin{equation}\label{1d.pde}
\left\{\begin{aligned}
&-(a(x)u'(x))'=f(x), &&\quad x\in(0,1),\\[3pt]
&u(0)=\alpha,\quad u(1)=\beta.
\end{aligned}\right.
\end{equation}

\paragraph{FD discretization}
We consider the discretization of \eqref{1d.pde} by the classical second-order central FD scheme on a uniform grid. In the case where $a(x)$ is constant, this is also known as the $(-1,2,-1)$ scheme. Let us describe it shortly; for more details on FD methods, we refer the reader to the available literature (see \cite{Smith} or any good book on FDs). Choose a discretization parameter $n\in\mathbb N$, set $h=\frac1{n+1}$ and $x_j=jh$ for all $j\in[0,n+1]$. For $j=1,\ldots,n$, we approximate $-(a(x)u'(x))'|_{x=x_j}$ by the classical second-order central FD formula:
\begin{align}\label{FD-formula}
-(a(x)u'(x))'|_{x=x_j}&\approx-\frac{a(x_{j+\frac1{2^{\vphantom{\mbox{\tiny 1}}}}})u'(x_{j+\frac1{2^{\vphantom{\mbox{\tiny 1}}}}})-a(x_{j-\frac1{2^{\vphantom{\mbox{\tiny 1}}}}})u'(x_{j-\frac1{2^{\vphantom{\mbox{\tiny 1}}}}})}{h}\notag\\[5pt]
&\approx-\frac{a(x_{j+\frac1{2^{\vphantom{\mbox{\tiny 1}}}}})\dfrac{u(x_{j+1})-u(x_j)}{h}-a(x_{j-\frac1{2^{\vphantom{\mbox{\tiny 1}}}}})\dfrac{u(x_j)-u({x_{j-1}})}{h}}{h}\notag\\[5pt]
&=\frac{-a(x_{j+\frac1{2^{\vphantom{\mbox{\tiny 1}}}}})u(x_{j+1})+\bigl(a(x_{j+\frac1{2^{\vphantom{\mbox{\tiny 1}}}}})+a(x_{j-\frac1{2^{\vphantom{\mbox{\tiny 1}}}}})\bigr)u(x_j)-a(x_{j-\frac1{2^{\vphantom{\mbox{\tiny 1}}}}})u(x_{j-1})}{h^2}.
\end{align}
This means that the values of the solution $u$ at the nodes $x_j$, $j=1,\ldots,n$, satisfy (approximately) the following linear system:
\begin{equation*}
-a(x_{j+\frac1{2^{\vphantom{\mbox{\tiny 1}}}}})u(x_{j+1})+\bigl(a(x_{j+\frac1{2^{\vphantom{\mbox{\tiny 1}}}}})+a(x_{j-\frac1{2^{\vphantom{\mbox{\tiny 1}}}}})\bigr)u(x_j)-a(x_{j-\frac1{2^{\vphantom{\mbox{\tiny 1}}}}})u(x_{j-1})=h^2f(x_j),\qquad j=1,\ldots,n.
\end{equation*}
We therefore approximate 
the nodal value $u(x_j)$ with the value $u_j$ 
for $j=0,\ldots,n+1$, where $u_0=\alpha$, $u_{n+1}=\beta$, and $\mathbf{u}=(u_1,\ldots,u_n)^T$ solves
\begin{equation}\label{j-form}
-a(x_{j+\frac1{2^{\vphantom{\mbox{\tiny1}}}}})u_{j+1}+\bigl(a(x_{j+\frac1{2^{\vphantom{\mbox{\tiny 1}}}}})+a(x_{j-\frac1{2^{\vphantom{\mbox{\tiny 1}}}}})\bigr)u_j-a(x_{j-\frac1{2^{\vphantom{\mbox{\tiny 1}}}}})u_{j-1}=h^2f(x_j),\qquad j=1,\ldots,n.
\end{equation}
The matrix of the linear system \eqref{j-form} is the $n\times n$ tridiagonal symmetric matrix given by
\begin{equation}\label{An}
A_n = \begin{bmatrix}
a_{\frac1{2^{\vphantom{\mbox{\tiny 1}}}}}+a_{\frac3{2^{\vphantom{\mbox{\tiny 1}}}}} & \ \ -a_{\frac3{2^{\vphantom{\mbox{\tiny 1}}}}} & \ \ \ \ & \ \ \ \ & \ \ \ \ \\[10pt]
-a_{\frac3{2^{\vphantom{\mbox{\tiny 1}}}}} & \ \ a_{\frac3{2^{\vphantom{\mbox{\tiny 1}}}}}+a_{\frac5{2^{\vphantom{\mbox{\tiny 1}}}}} & \ \ \ \ -a_{\frac5{2^{\vphantom{\mbox{\tiny 1}}}}} & \ \ \ \ & \ \ \ \ \\[10pt]
& \ \ -a_{\frac5{2^{\vphantom{\mbox{\tiny 1}}}}} & \ \ \ \ \ddots & \ \ \ \ \ddots & \ \ \ \ \\[10pt]
& \ \ & \ \ \ \ \ddots & \ \ \ \ \ddots & \ \ \ \ -a_{n-\frac1{2^{\vphantom{\mbox{\tiny 1}}}}} \\[10pt]
& \ \ & \ \ \ \ & \ \ \ \ -a_{n-\frac1{2^{\vphantom{\mbox{\tiny 1}}}}} & \ \ \ \ a_{n-\frac1{2^{\vphantom{\mbox{\tiny 1}}}}}+a_{n+\frac1{2^{\vphantom{\mbox{\tiny 1}}}}}
\end{bmatrix},
\end{equation}
where $a_i=a(x_i)$ for all $i\in[0,n+1]$. 

\paragraph{GLT spectral analysis of the FD discretization matrices}
We are going to see that the theory of GLT sequences allows one to compute the singular value and spectral distribution of the sequence of FD discretization matrices $\{A_n\}_n$ under a very weak hypothesis on the coefficient $a(x)$. 

\begin{theorem}\label{FD_T1}
If $a:[0,1]\to\mathbb R$ is continuous a.e.\ then
\begin{equation}\label{diffGLT}
\{A_n\}_n\sim_{\rm GLT}a(x)(2-2\cos\theta)
\end{equation}
and
\begin{equation}\label{diffsigla}
\{A_n\}_n\sim_{\sigma,\lambda}a(x)(2-2\cos\theta).
\end{equation}
\end{theorem}
\begin{proof}
It suffices to prove \eqref{diffGLT} because \eqref{diffsigla} follows from \eqref{diffGLT} and {\bf GLT1} as the matrices $A_n$ are symmetric.
Write
\[ A_n=D_n^+K_n^++D_n^-K_n^-, \]
where
\begin{align*}
K_n^+&= \left[\begin{array}{ccccc}
1 & -1 & & & \\
& 1 & -1 & & \\
& & \ddots & \ddots & \\
& & & 1 & -1\\
& & & & 1
\end{array}\right]=T_n(1-{\rm e}^{-{\rm i}\theta}), &K_n^-&=\left[\begin{array}{ccccc}
1 & & & & \\
-1 & 1 & & & \\
& \ddots & \ddots & & \\
& & -1 & 1 & \\
& & & -1 & 1
\end{array}\right]=T_n(1-{\rm e}^{{\rm i}\theta}),\\[5pt]
D_n^+&=\mathop{\rm diag}_{j=1,\ldots,n}a_{j+\frac1{2^{\vphantom{\mbox{\tiny 1}}}}}=\mathop{\rm diag}_{j=1,\ldots,n}a(x_{j+\frac1{2^{\vphantom{\mbox{\tiny 1}}}}}), &D_n^-&=\mathop{\rm diag}_{j=1,\ldots,n}a_{j-\frac1{2^{\vphantom{\mbox{\tiny 1}}}}}=\mathop{\rm diag}_{j=1,\ldots,n}a(x_{j-\frac1{2^{\vphantom{\mbox{\tiny 1}}}}}).\notag
\end{align*}
It is easy to check that the grids $\mathcal G_n^+=\{x_{j+\frac1{2^{\vphantom{\mbox{\tiny 1}}}}}\}_{j=1,\ldots,n}$ and $\mathcal G_n^-=\{x_{j-\frac1{2^{\vphantom{\mbox{\tiny 1}}}}}\}_{j=1,\ldots,n}$ are a.u.\ in $[0,1]$. Hence, by {\bf GLT3}, $\{D_n^+\}_n\sim_{\rm GLT}a(x)$ and $\{D_n^-\}_n\sim_{\rm GLT}a(x)$. We then infer from {\bf GLT3}--{\bf GLT4} that
\begin{equation*}
\{A_n\}_n\sim_{\rm GLT}a(x)(1-{\rm e}^{-{\rm i}\theta})+a(x)(1-{\rm e}^{{\rm i}\theta})=a(x)(2-2\cos\theta). \tag*{\qedhere}
\end{equation*}
\end{proof}

\begin{remark}[\textbf{formal structure of the symbol}]\label{oss1}
From a formal point of view (i.e., disregarding the regularity of $a(x)$ and $u(x)$), problem \eqref{1d.pde} can be rewritten in the form
\begin{equation*} 
\left\{\begin{aligned}
&-a(x)u''(x)-a'(x)u'(x)=f(x), &&\quad x\in(0,1),\\[3pt]
&u(0)=\alpha,\quad u(1)=\beta.
\end{aligned}\right.
\end{equation*}
From this reformulation, it appears more clearly that the symbol $a(x)(2-2\cos\theta)$ consists of two ``ingredients'':
\begin{itemize}[nolistsep,leftmargin=*]
	\item The coefficient of the higher-order differential operator, namely $a(x)$, in the physical variable $x$. To make a parallelism with H\"ormander's theory \cite{hormander}, the higher-order differential operator $-a(x)u''(x)$ is the so-called principal symbol of the complete differential operator $-a(x)u''(x)-a'(x)u'(x)$ and $a(x)$ is then the coefficient of the principal symbol.
	\item The trigonometric polynomial associated with the FD formula $(-1,2,-1)$ used to approximate the higher-order derivative $-u''(x)$, namely $2-2\cos\theta={-{\rm e}^{-{\rm i}\theta}+2-{\rm e}^{{\rm i}\theta}}$, in the Fourier variable $\theta$. To see that $(-1,2,-1)$ is precisely the FD formula used to approximate $-u''(x)$, simply imagine $a(x)=1$ and note that in this case the FD formula \eqref{FD-formula} becomes
	\[ -u''(x_j)\approx\frac{-u(x_{j+1})+2u(x_j)-u(x_{j-1})}{h^2}, \]
	i.e., the FD formula $(-1,2,-1)$ to approximate $-u''(x_j)$.
\end{itemize}
We observe that the term $-a'(x)u'(x)$, which only depends on lower-order derivatives of $u(x)$, does not enter the expression of the symbol.
\end{remark}

\begin{remark}[\textbf{nonnegativity and order of the zero at $\boldsymbol{\theta=0}$}]\label{oss1.1}
The trigonometric polynomial $2-2\cos\theta$ is nonnegative on $[-\pi,\pi]$ and it has a unique zero of order 2 at $\theta=0$, because 
\[ \lim_{\theta\to0}\frac{2-2\cos\theta}{\theta^2}=1. \]
This reflects the fact that the associated FD formula $(-1,2,-1)$ approximates $-u''(x)$, which is a differential operator of order 2 (it is also nonnegative on the space of functions $v\in C^2([0,1])$ such that $v(0)=v(1)=0$, in the sense that $\int_0^1-v''(x)v(x){\rm d} x=\int_0^1(v'(x))^2{\rm d} x\ge0$ for all such $v$).
\end{remark}

\begin{example}\label{exa}
\begin{figure}
\centering
\begin{minipage}{0.6\textwidth}
\centering
\includegraphics[width=\textwidth]{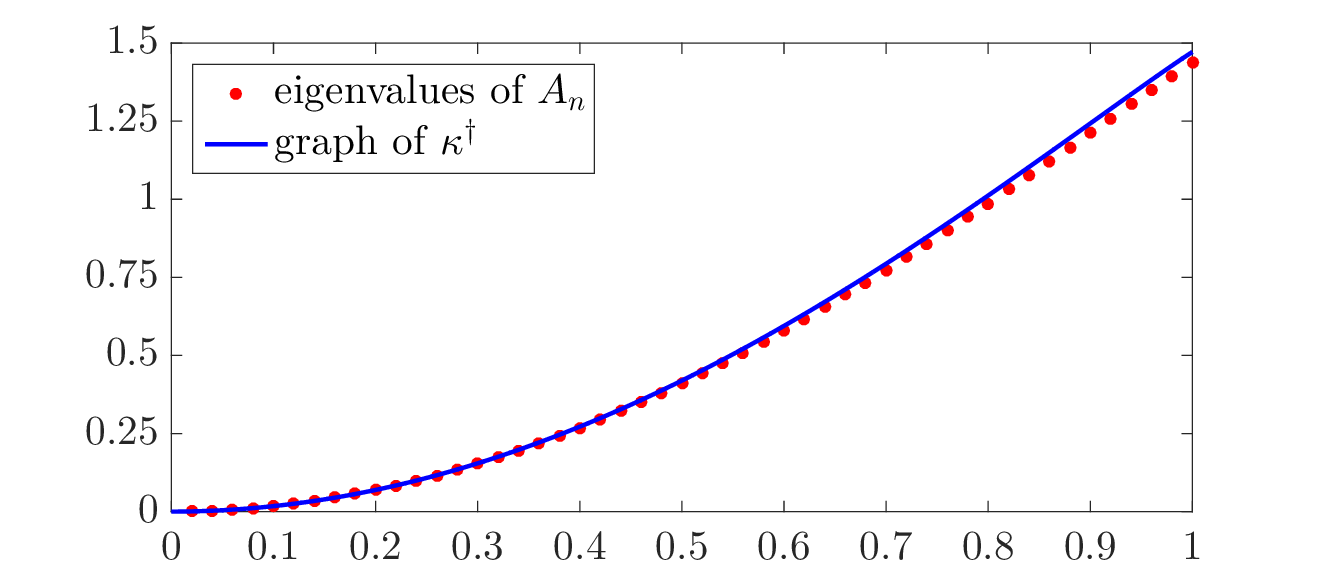}
\caption{Example~\ref{exa}: Comparison between the spectrum of $A_n$ and the monotone rearrangement $\kappa^\dag$ of the symbol $a(x)(2-2\cos\theta)$ for $n=50$ and $a(x)=x{\rm e}^{-x}$.}
\label{FDd2}
\end{minipage}
\hspace{0.01\textwidth}
\begin{minipage}{0.375\textwidth}
\centering
\captionof{table}{Example~\ref{exa}: Computation of $|\mathbf{s}_n-\mathbf{e}_n|_\infty$ for $a(x)=x{\rm e}^{-x}$ and for increasing values of $n$.}
\begin{tabular}{rc}
\toprule
$n$ & $|\mathbf{s}_n-\mathbf{e}_n|_\infty$\\
\midrule
50   & 0.0327\\
100  & 0.0165\\
200  & 0.0083\\
400  & 0.0042\\
800 & 0.0022\\
1600 & 0.0011\\
\bottomrule
\end{tabular}
\label{FDdT2}
\end{minipage}
\end{figure}
According to the informal meaning of the spectral distribution $\{A_n\}_n\sim_\lambda a(x)(2-2\cos\theta)$ in \eqref{diffsigla}, if $n$ is large enough, then, assuming we have no outliers, the eigenvalues of $A_n$ are approximately equal to the uniform samples $\kappa^\dag(\frac in)$, $i=1,\ldots,n$, where $\kappa^\dag:[0,1]\to\mathbb R$ is the monotone rearrangement of the spectral symbol $\kappa(x,\theta)=a(x)(2-2\cos\theta):[0,1]\times[0,\pi]\to\mathbb R$.\,\footnote{\,Note that, by Definition~\ref{def:distr}, both $\kappa(x,\theta):[0,1]\times[-\pi,\pi]\to\mathbb R$ and its restriction to $[0,1]\times[0,\pi]$ considered here are spectral symbols for $\{A_n\}_n$, due to the symmetry of $\kappa(x,\theta)$ with respect to the variable $\theta$.} 
This is confirmed by Figure~\ref{FDd2} and Table~\ref{FDdT2} for the case $a(x)=x{\rm e}^{-x}$.
In Figure~\ref{FDd2}, we plotted the eigenvalues of $A_n$ for $n=50$ and the graph of $\kappa^\dag$. The eigenvalues of $A_n$ were sorted in ascending order and positioned at $\frac in$, $i=1,\ldots,n$. The graph of $\kappa^\dag$ was obtained through the procedure {\bf R1}, i.e., by plotting for $r=1000$ the graph of the piecewise linear function $\kappa^\dag_r:[0,1]\to\mathbb R$ that interpolates over the nodes $(0,\frac1{r^{2^{\vphantom{0}}}},\frac2{r^{2^{\vphantom{0}}}},\ldots,1)$ the samples $\kappa(\frac{i_{1_{\vspace{0pt}}}}r,\frac{i_{2_{\vspace{0pt}}}\pi}r)$, $i_1,i_2=1,\ldots,r$, arranged in ascending order.
We observe an excellent agreement between the eigenvalues of $A_n$ and the graph of $\kappa^\dag$ (which also indicates the absence of outliers in this case).
In Table~\ref{FDdT2}, we computed, for increasing values of $n$, the $\infty$-norm of the difference $\mathbf{s}_n-\mathbf{e}_n$, where
\begin{itemize}[nolistsep,leftmargin=*]
	\item $\mathbf{e}_n$ is the vector of the eigenvalues $\lambda_i(A_n)$, $i=1,\ldots,n$, sorted in ascending order,
	\item $\mathbf{s}_n$ is the vector of the samples $\kappa^\dag(\frac in)$, $i=1,\ldots,n$, sorted in ascending order.
\end{itemize}
As shown in Table~\ref{FDdT2}, the norm $|\mathbf{s}_n-\mathbf{e}_n|_\infty$ converges to~$0$ as $n\to\infty$, though the convergence is slow.
We point out that in Table~\ref{FDdT2} we approximated $\kappa^\dag$ by $\kappa^\dag_r$ with $r=5000$ instead of $r=1000$, so as to obtain more accurate values for the samples $\kappa^\dag(\frac in)$, $i=1,\ldots,n$.
\end{example}


\subsubsection{FD Discretization of Convection-Diffusion-Reaction Equations}\label{oss2}

\subsubsection*{1st Part}
Suppose we add to the diffusion equation \eqref{1d.pde} a convection and a reaction term. In this way, we obtain the following convection-diffusion-reaction equation in divergence form with Dirichlet boundary conditions:
\begin{equation}\label{1d.pde''}
\left\{\begin{aligned}
&-(a(x)u'(x))'+b(x)u'(x)+c(x)u(x)=f(x), &&\quad x\in(0,1),\\[3pt]
&u(0)=\alpha,\quad u(1)=\beta.
\end{aligned}\right.
\end{equation}
Based on Remark~\ref{oss1}, we expect that the term $b(x)u'(x)+c(x)u(x)$, which only involves lower-order derivatives of $u(x)$, does not enter the expression of the symbol. In other words, if we discretize the higher-order term $-(a(x)u'(x))'$ as in \eqref{FD-formula}, the symbol of the resulting FD discretization matrices $B_n$ should be again $a(x)(2-2\cos\theta)$. We are going to show that this is in fact the case. This highlights a general aspect: {\em lower-order terms such as $b(x)u'(x)+c(x)u(x)$ do not enter the expression of the symbol and do not affect in any way the singular value and spectral distribution of DE discretization matrices.}

\paragraph{FD discretization}
Let $n\in\mathbb N$, set $h=\frac1{n+1}$ and $x_j=jh$ for all $j\in[0,n+1]$. Consider the discretization of \eqref{1d.pde''} by the FD scheme defined as follows.
\begin{itemize}[nolistsep,leftmargin=*]
 	\item To approximate the higher-order (diffusion) term $-(a(x)u'(x))'$, use again the FD formula \eqref{FD-formula}, i.e.,
 	\begin{equation}\label{FD-formula-repeat}
 	-(a(x)u'(x))'|_{x=x_j}\approx\frac{-a(x_{j+\frac1{2^{\vphantom{\mbox{\tiny 1}}}}})u(x_{j+1})+\bigl(a(x_{j+\frac1{2^{\vphantom{\mbox{\tiny 1}}}}})+a(x_{j-\frac1{2^{\vphantom{\mbox{\tiny 1}}}}})\bigr)u(x_j)-a(x_{j-\frac1{2^{\vphantom{\mbox{\tiny 1}}}}})u(x_{j-1})}{h^2}.
 	\end{equation}
 	\item To approximate the convection term $b(x)u'(x)$, use any 
 	FD formula; to fix ideas, here we use the second-order central formula
\begin{equation}\label{FD-ordine1}
b(x)u'(x)|_{x=x_j}\approx b(x_j)\frac{u(x_{j+1})-u(x_{j-1})}{2h}.
\end{equation}
	\item To approximate the reaction term $c(x)u(x)$, use the obvious equation
	\begin{equation}\label{f.re}
	c(x)u(x)|_{x=x_j}=c(x_j)u(x_j).
	\end{equation}
\end{itemize}
This means that the values of the solution $u$ at the nodes $x_j$, $j=1,\ldots,n$, satisfy (approximately) the following linear system:
\begin{align*}
&-a(x_{j+\frac1{2^{\vphantom{\mbox{\tiny 1}}}}})u(x_{j+1})+\bigl(a(x_{j+\frac1{2^{\vphantom{\mbox{\tiny 1}}}}})+a(x_{j-\frac1{2^{\vphantom{\mbox{\tiny 1}}}}})\bigr)u(x_j)-a(x_{j-\frac1{2^{\vphantom{\mbox{\tiny 1}}}}})u(x_{j-1})\notag\\
&+\frac h2\bigl(b(x_j)u(x_{j+1})-b(x_j)u(x_{j-1})\bigr)+h^2c(x_j)u(x_j)=h^2f(x_j),\qquad j=1,\ldots,n.
\end{align*}
We therefore approximate the nodal value $u(x_j)$ with the value $u_j$ for $j=0,\ldots,n+1$, where $u_0=\alpha$, $u_{n+1}=\beta$, and $\mathbf{u}=(u_1,\ldots,u_n)^T$ solves
\begin{align*}
&-a(x_{j+\frac1{2^{\vphantom{\mbox{\tiny 1}}}}})u_{j+1}+\bigl(a(x_{j+\frac1{2^{\vphantom{\mbox{\tiny 1}}}}})+a(x_{j-\frac1{2^{\vphantom{\mbox{\tiny 1}}}}})\bigr)u_j-a(x_{j-\frac1{2^{\vphantom{\mbox{\tiny 1}}}}})u_{j-1}\notag\\
&+\frac h2\bigl(b(x_j)u_{j+1}-b(x_j)u_{j-1}\bigr)+h^2c(x_j)u_j=h^2f(x_j),\qquad j=1,\ldots,n.
\end{align*}
The matrix $B_n$ of this linear system
admits a natural decomposition as 
\begin{equation}\label{conv-reac}
B_n=A_n+Z_n,
\end{equation}
where $A_n$ is the matrix coming from the discretization of the higher-order (diffusion) term $-(a(x)u'(x))$, while $Z_n$ is the matrix coming from the discretization of the lower-order (convection and reaction) terms $b(x)u'(x)$ and $c(x)u(x)$. Note that $A_n$ is given by \eqref{An} and $Z_n$ is given by 
\begin{align}\label{Y_n-cdr}
Z_n&=\frac h2\begin{bmatrix}
0 & b_1 & & & \\[5pt]
-b_2 & 0 & b_2 & & \\[5pt]
& \ddots & \ddots & \ddots & \\[5pt]
& & -b_{n-1} & 0 & b_{n-1}\\[5pt]
& & & -b_n & 0
\end{bmatrix}+h^2\begin{bmatrix}
c_1 & & & & \\[5pt]
& c_2 & & & \\[5pt]
& & \ddots & & \\[5pt]
& & & c_{n-1} & \\[5pt]
& & & & c_n
\end{bmatrix},
\end{align}
where $b_i=b(x_i)$ and $c_i=c(x_i)$ for all $i=1,\ldots,n$.

\paragraph{GLT spectral analysis of the FD discretization matrices}
We now prove that Theorem~\ref{FD_T1} holds unchanged with $B_n$ in place of $A_n$. 

\begin{theorem}\label{FD_T2}
If $a:[0,1]\to\mathbb R$ is continuous a.e.\ and $b,c:[0,1]\to\mathbb R$ are bounded then
\begin{equation}\label{cdr1GLT}
\{B_n\}_n\sim_{\rm GLT}a(x)(2-2\cos\theta)
\end{equation}
and
\begin{equation}\label{cdr1sigla}
\{B_n\}_n\sim_{\sigma,\lambda}a(x)(2-2\cos\theta).
\end{equation}
\end{theorem}
\begin{proof}
The matrix $Z_n$ in \eqref{Y_n-cdr} satisfies
\begin{equation}\label{Z0d}
\|Z_n\|_2\le\left(2(n-1)\|b\|_\infty^2\frac{h^2}4\right)^{1/2}+\left(n\|c\|_\infty^2h^4\right)^{1/2}=2^{1/2}(n-1)^{1/2}\|b\|_\infty\frac h2+n^{1/2}\|c\|_\infty h^2=o(n^{1/2}).
\end{equation}
As a consequence, $\{Z_n\}_n\sim_\sigma0$ by {\bf Z2} and $\{Z_n\}_n\sim_{\rm GLT}0$ by {\bf GLT3}.
Since $\{A_n\}_n\sim_{\rm GLT}a(x)(2-2\cos\theta)$ by Theorem~\ref{FD_T1}, the decomposition $B_n=A_n+Z_n$ in \eqref{conv-reac} and {\bf GLT4} imply \eqref{cdr1GLT}.

Now, if the convection term is constant, i.e., $b(x)=C$ identically for some constant $C$, then $B_n$ is symmetric and \eqref{cdr1sigla} follows from \eqref{cdr1GLT} and {\bf GLT1}. If $b(x)$ is not constant, then $B_n$ is not symmetric in general and so \eqref{cdr1GLT} and {\bf GLT1} only imply the singular value distribution $\{B_n\}_n\sim_\sigma a(x)(2-2\cos\theta)$.
However, also the spectral distribution $\{B_n\}_n\sim_\lambda a(x)(2-2\cos\theta)$ holds by \eqref{cdr1GLT} and {\bf GLT2} in view of the decomposition $B_n=A_n+Z_n$, since $A_n$ is symmetric and $\|Z_n\|_2=o(n^{1/2})$ by \eqref{Z0d}.
\end{proof}



\subsubsection*{2nd Part}
So far, we only considered DEs with Dirichlet boundary conditions. A natural question is the following: if we change the boundary conditions, does the expression of the symbol change? The answer is ``no'': {\em boundary conditions do not affect the singular value and eigenvalue distribution because they only produce a small-rank perturbation in the resulting discretization matrices}. To better understand this point, we consider problem \eqref{1d.pde''} with Neumann boundary conditions:
\begin{equation}\label{Neumann}
\left\{\begin{aligned}
&-(a(x)u'(x))'+b(x)u'(x)+c(x)u(x)=f(x), &&\quad x\in(0,1),\\[3pt]
&u'(0)=\alpha,\quad u'(1)=\beta.
\end{aligned}\right.
\end{equation}

\paragraph{FD discretization}
We discretize \eqref{Neumann} by the same FD scheme considered in the 1st part, which is defined by the FD formulas \eqref{FD-formula-repeat}--\eqref{f.re}. In this way, we arrive at the linear system
\begin{equation}\label{complete.sys}
\begin{aligned}
&-a(x_{j+\frac1{2^{\vphantom{\mbox{\tiny 1}}}}})u_{j+1}+\bigl(a(x_{j+\frac1{2^{\vphantom{\mbox{\tiny 1}}}}})+a(x_{j-\frac1{2^{\vphantom{\mbox{\tiny 1}}}}})\bigr)u_j-a(x_{j-\frac1{2^{\vphantom{\mbox{\tiny 1}}}}})u_{j-1}\\
&+\frac h2\bigl(b(x_j)u_{j+1}-b(x_j)u_{j-1}\bigr)+h^2c(x_j)u_j=h^2f(x_j),\qquad j=1,\ldots,n,
\end{aligned}
\end{equation}
which is formed by $n$ equations in $n+2$ unknowns $u_0,u_1,\ldots,u_n,u_{n+1}$. Note that $u_0$ and $u_{n+1}$ should now be considered as unknowns, because they are not specified by the Dirichlet boundary conditions. However, as it is common in the FD context, $u_0$ and $u_{n+1}$ are expressed in terms of $u_1,\ldots,u_n$ by exploiting the Neumann boundary conditions. The simplest choice is to express $u_0$ and $u_{n+1}$ as a function of $u_1$ and $u_n$, respectively, by imposing the conditions
\begin{equation}\label{dBCs}
\frac{u_1-u_0}{h}=\alpha,\qquad\frac{u_{n+1}-u_n}{h}=\beta,
\end{equation}
which yield $u_0=u_1-\alpha h$ and $u_{n+1}=u_n+\beta h$. Substituting into \eqref{complete.sys}, we obtain a linear system with $n$ equations and $n$ unknowns $u_1,\ldots,u_n$. Setting ${a_i=a(x_i)}$, ${b_i=b(x_i)}$, ${c_i=c(x_i)}$ for all $i\in[0,n+1]$, the matrix of this system is
\begin{equation}\label{C-eq}
C_n=B_n+R_n=A_n+Z_n+R_n,
\end{equation}
where $A_n$, $B_n$, $Z_n$ are given by \eqref{An}, \eqref{conv-reac}, \eqref{Y_n-cdr}, respectively, and
\[ R_n=\begin{bmatrix}
-a_{\frac1{2^{\vphantom{\mbox{\tiny 1}}}}}-\dfrac{h}{2^{\vphantom{\mbox{\tiny 1}}}}b_1 & & & & \\
& & & & \\
& & & & \\
& & & & -a_{n+\frac{1}{2^{\vphantom{\mbox{\tiny 1}}}}}+\dfrac{h}{2^{\vphantom{\mbox{\tiny 1}}}}b_n
\end{bmatrix} \]
is a small-rank correction coming from the discretization \eqref{dBCs} of the boundary conditions.

\paragraph{GLT spectral analysis of the FD discretization matrices}
We prove that Theorems~\ref{FD_T1} and~\ref{FD_T2} hold unchanged with $C_n$ in place of $A_n$ and $B_n$, respectively.
\begin{theorem}\label{FD_T3}
If $a:[0,1]\to\mathbb R$ is continuous a.e.\ and $b,c:[0,1]\to\mathbb R$ are bounded then
\begin{equation}\label{cdr2GLT}
\{C_n\}_n\sim_{\rm GLT}a(x)(2-2\cos\theta)
\end{equation}
and
\begin{equation}\label{cdr2sigla}
\{C_n\}_n\sim_{\sigma,\lambda}a(x)(2-2\cos\theta).
\end{equation}
\end{theorem}
\begin{proof}
The matrix $R_n$ satisfies
\[ \|R_n\|_2^2\le2\left(\|a\|_\infty+\frac{h}2\|b\|_\infty\right)^2=o(n). \]
Recalling that $\|Z_n\|_2=o(n^{1/2})$ by \eqref{Z0d}, we have
\[ \|Z_n+R_n\|_2\le\|Z_n\|_2+\|R_n\|_2=o(n^{1/2}). \]
Hence, $\{Z_n+R_n\}_n\sim_\sigma0$ by {\bf Z2} and $\{Z_n+R_n\}_n\sim_{\rm GLT}0$ by ${\bf GLT3}$. Since $\{A_n\}_n\sim_{\rm GLT}a(x)(2-2\cos\theta)$ by Theorem~\ref{FD_T1}, the decomposition $C_n=A_n+Z_n+R_n$ in \eqref{C-eq} and {\bf GLT4} imply \eqref{cdr2GLT}. The singular value distribution in \eqref{cdr2sigla} follows from \eqref{cdr2GLT} and {\bf GLT1}. The spectral distribution in \eqref{cdr2sigla} follows from \eqref{cdr2GLT} and {\bf GLT2} in view of the decomposition $C_n=A_n+Z_n+R_n$, since $A_n$ is symmetric and $\|Z_n+R_n\|_2=o(n^{1/2})$.
%
\end{proof}

\subsubsection*{3rd Part}
Consider the following convection-diffusion-reaction problem:
\begin{equation}\label{non-div}
\left\{\begin{aligned}
&-a(x)u''(x)+b(x)u'(x)+c(x)u(x)=f(x), &&\quad x\in(0,1),\\[3pt]
&u(0)=\alpha,\quad u(1)=\beta.
\end{aligned}\right.
\end{equation}
The difference with respect to problem \eqref{1d.pde''} is that the higher-order differential operator now appears in non-divergence form, i.e., we have $-a(x)u''(x)$ instead of $-(a(x)u'(x))'$. Nevertheless, based on Remark~\ref{oss1}, if we use again the FD formula $(-1,2,-1)$ to discretize the second derivative $-u''(x)$, the symbol of the resulting FD discretization matrices should be again $a(x)(2-2\cos\theta)$. We are going to show that this is in fact the case.

\paragraph{FD discretization}
Let $n\in\mathbb N$, set $h=\frac1{n+1}$ and $x_j=jh$ for all $j=0,\ldots,n+1$. We discretize \eqref{non-div} by using again the central second-order FD scheme, which in this case is defined by the following formulas: 
\begin{align*}
-a(x)u''(x)|_{x=x_j}&\approx a(x_j)\frac{-u(x_{j+1})+2u(x_j)-u(x_{j-1})}{h^2},\qquad j=1,\ldots,n,\\
b(x)u'(x)|_{x=x_j}&\approx b(x_j)\frac{u(x_{j+1})-u(x_{j-1})}{2h},\qquad j=1,\ldots,n,\\
c(x)u(x)|_{x=x_j}&=c(x_j)u(x_j),\qquad j=1,\ldots,n.
\end{align*}
Then, we approximate the nodal value $u(x_j)$ with the value $u_j$ for $j=0,\ldots,n+1$, 
where $u_0=\alpha$, $u_{n+1}=\beta$, and $\mathbf{u}=(u_1,\ldots,u_n)^T$ solves
\begin{align*}
a(x_j)(-u_{j+1}+2u_j-u_{j-1})+\frac h2b(x_j)(u_{j+1}-u_{j-1})+h^2c(x_j)u_j=h^2f(x_j),\qquad j=1,\ldots,n.
\end{align*}
The matrix $E_n$ of this linear system can be decomposed according to the diffusion, convection and reaction term, as follows:
\begin{equation}\label{A-dcr}
E_n=K_n+Z_n,
\end{equation}
where $Z_n$ is the sum of the convection and reaction matrix and is given by \eqref{Y_n-cdr}, while
\begin{align}
K_n&=\begin{bmatrix}
2a_1 & \ -a_1 & \ & \ & \ \\[5pt]
-a_2 & \ 2a_2 & \ -a_2 & \ & \ \\[5pt]
& \ \ddots & \ \ddots & \ \ddots & \ \\[5pt]
& \ & \ -a_{n-1} & \ 2a_{n-1} & \ -a_{n-1}\\[5pt]
& \ & \ & \ -a_n & \ 2a_n
\end{bmatrix}=\left(\mathop{\rm diag}_{i=1,\ldots,n}a_i\right)T_n(2-2\cos\theta)\label{K.0}
\end{align}
is the diffusion matrix ($a_i=a(x_i)$ for all $i=1,\ldots,n$). 

\paragraph{GLT spectral analysis of the FD discretization matrices}
Despite the nonsymmetry of the diffusion matrix, which is due to the non-divergence form of the higher-order (diffusion) operator $-a(x)u''(x)$, we prove that Theorems~\ref{FD_T1}, \ref{FD_T2}, \ref{FD_T3} hold unchanged with $E_n$ in place of $A_n$, $B_n$, $C_n$, respectively. The only difference is that now we have to strengthen the hypothesis on the coefficient $a(x)$.

\begin{theorem}\label{FD_T4}
If $a:[0,1]\to\mathbb R$ is continuous and $b,c:[0,1]\to\mathbb R$ are bounded then
\begin{equation}\label{cdr3GLT}
\{E_n\}_n\sim_{\rm GLT}a(x)(2-2\cos\theta)
\end{equation}
and
\begin{equation}\label{cdr3sigla}
\{E_n\}_n\sim_{\sigma,\lambda}a(x)(2-2\cos\theta).
\end{equation}
\end{theorem}
\begin{proof}
We know from \eqref{Z0d} that
\[ \|Z_n\|_2=o(n^{1/2}), \]
and so $\{Z_n\}_n\sim_\sigma0$ by {\bf Z2} and $\{Z_n\}_n\sim_{\rm GLT}0$ by {\bf GLT3}.
Since the grid $\mathcal G_n=\{x_j\}_{j=1,\ldots,n}$ is clearly a.u.\ in $[0,1]$, by \eqref{K.0} and {\bf GLT3}--{\bf GLT4} we have
\begin{equation*}
\{K_n\}_n\sim_{\rm GLT}a(x)(2-2\cos\theta).
\end{equation*}
Hence, \eqref{cdr3GLT} follows from {\bf GLT4} and the decomposition \eqref{A-dcr}.
From \eqref{cdr3GLT} and {\bf GLT1} we obtain the singular value distribution in \eqref{cdr3sigla}. Note that so far we have not used the continuity assumption on $a$, i.e., both \eqref{cdr3GLT} and the singular value distribution in \eqref{cdr3sigla} hold under the weaker assumption that $a$ is continuous a.e.
To obtain the spectral distribution in \eqref{cdr3sigla}, the idea is to exploit the fact that $K_n$ is ``almost'' symmetric, because $a(x)$ varies continuously when $x$ ranges in $[0,1]$, and so $a(x_j)\approx a(x_{j+1})$ for all $j=1,\ldots,n-1$ (when $n$ is large enough). Therefore, by replacing $K_n$ with one of its symmetric approximations $\tilde K_n$, we can write
\begin{equation}\label{Ktilde-dec}
E_n=\tilde K_n+(K_n-\tilde K_n)+Z_n,
\end{equation}
and we want to obtain the spectral distribution in \eqref{cdr3sigla} from \eqref{cdr3GLT} and {\bf GLT2} applied with $X_n=\tilde K_n$ and $Y_n=(K_n-\tilde K_n)+Z_n$.
Let 
\begin{equation}\label{K_tilde}
\tilde K_n=\begin{bmatrix}
2a_1 & \ -a_1 & \ & \ & \ \\[5pt]
-a_1 & \ 2a_2 & \ -a_2 & \ & \ \\[5pt]
& \ \ddots & \ \ddots & \ \ddots & \ \\[5pt]
& \ & \ -a_{n-2} & \ 2a_{n-1} & \ -a_{n-1}\\[5pt]
& \ & \ & \ -a_{n-1} & \ 2a_n
\end{bmatrix}.
\end{equation}
We have
\begin{align*}
&\|K_n-\tilde K_n\|_2^2=\sum_{\ell=1}^{n-1}(a_{\ell+1}-a_\ell)^2=\sum_{\ell=1}^{n-1}(a(x_{\ell+1})-a(x_\ell))^2\le(n-1)\,\omega_a(h)^2=o(n)
\end{align*}
and
\[ \|(K_n-\tilde K_n)+Z_n\|_2\le\|K_n-\tilde K_n\|_2+\|Z_n\|_2=o(n^{1/2}). \]
Thus, {\bf GLT2} implies $\{E_n\}_n\sim_\lambda a(x)(2-2\cos\theta)$.
\end{proof}

\begin{remark}\label{Dt}
The matrix $\tilde K_n$ used in the proof of Theorem~\ref{FD_T4} is nothing else than $S_n^{\mathcal G_n}(a)\circ T_n(2-2\cos\theta)$, where $S_n^{\mathcal G_n}(a)$ is the $n$th arrow-shaped sampling matrix generated by $a$ corresponding to the a.u.\ grid $\mathcal G_n=\{x_j\}_{j=1,\ldots,n}$; see Section~\ref{assm}.
\end{remark}

\subsubsection*{4th Part}
Based on Remark~\ref{oss1}, if we change the FD scheme to discretize 
the differential problem \eqref{non-div}, the symbol should become $a(x)p(\theta)$, where $p(\theta)$ is the trigonometric polynomial associated with the new FD formula used to approximate the second derivative $-u''(x)$ (the higher-order differential operator). We are going to show through an example that this is indeed the case.

\paragraph{FD discretization}
Consider the convection-diffusion-reaction problem \eqref{non-div}. Instead of the second-order central FD scheme $(-1,2,-1)$, this time we use the fourth-order central FD scheme $\frac1{12^{\vphantom{\mbox{\tiny1}}}}(1,-16,30,-16,1)$ to approximate the second derivative $-u''(x)$. 
In other words, for $j=2,\ldots,n-1$ we approximate the higher-order term $-a(x)u''(x)$ by the FD formula
$$-a(x)u''(x)|_{x=x_j}\approx a(x_j)\frac{u(x_{j+2})-16u(x_{j+1})+30u(x_j)-16u(x_{j-1})+u(x_{j-2})}{12h^2},$$
while for $j=1,n$ we use again the FD scheme $(-1,2,-1)$,
$$-a(x)u''(x)|_{x=x_j}\approx a(x_j)\frac{-u(x_{j+1})+2u(x_j)-u(x_{j-1})}{h^2}.$$
From a numerical point of view, this is not a good choice because the FD formula $\frac1{12^{\vphantom{\mbox{\tiny1}}}}(1,-16,30,-16,1)$ is a very accurate fourth-order formula, and in order not to destroy the accuracy one would gain from this formula, one should use a fourth-order scheme also for $j=1,n$ instead of the classical $(-1,2,-1)$. However, in this work we are not concerned with this kind of issues and we use the classical $(-1,2,-1)$ because it is simpler and allows us to better illustrate the GLT spectral analysis without introducing technicalities. 
As already observed before, the FD schemes used to approximate the lower-order terms $b(x)u'(x)$ and $c(x)u(x)$ do not affect the symbol, as well as the singular value and eigenvalue distribution, of the resulting sequence of discretization matrices. To illustrate once again this point, in this example we assume to approximate $b(x)u'(x)$ and $c(x)u(x)$ by the following alternative FD formulas: for $j=1,\ldots,n$,
\begin{align*}
b(x)u'(x)|_{x=x_j}&\approx b(x_j)\frac{u(x_j)-u(x_{j-1})}{h},\\[3pt]
c(x)u(x)|_{x=x_j}&\approx c(x_j)\frac{u(x_{j+1})+u(x_j)+u(x_{j-1})}3.
\end{align*}
Setting $a_i=a(x_i),\ b_i=b(x_i),\ c_i=c(x_i)$ for all $i=1,\ldots,n$, the resulting FD discretization matrix $P_n$ can be decomposed according to the diffusion, convection and reaction term, as follows:
\begin{equation*} 
P_n=K_n+Z_n,
\end{equation*}
where $Z_n$ is the sum of the convection and reaction matrix,
$$Z_n=h\begin{bmatrix}
b_1 & \ & \ & \ & \ \\[5pt]
-b_2 & \ b_2 & \ & \ & \ \\[5pt]
& \ \ddots & \ \ddots & \ & \ \\[5pt]
& \ & \ -b_{n-1} & \ b_{n-1}\\[5pt]
& \ & \ & \ -b_n & \ b_n
\end{bmatrix}+
\frac{h^2}3\begin{bmatrix}
c_1 & \ c_1 & \ & \ & \ \\[5pt]
c_2 & \ c_2 & \ c_2 & \ & \ \\[5pt]
& \ \ddots & \ \ddots & \ \ddots & \ \\[5pt]
& \ & \ c_{n-1} & \ c_{n-1} & \ c_{n-1}\\[5pt]
& \ & \ & \ c_n & \ c_n
\end{bmatrix},$$
while $K_n$ is the diffusion matrix,
\begin{align*}
K_n&\hspace{-1pt}=\hspace{-0.5pt}\frac1{12}\hspace{-1pt}\begin{bmatrix}
24a_1 & \ -12a_1 & \ & \ & \ & \ & \ \\[9pt]
-16a_2 & \ 30a_2 & \ -16a_2 & \ a_2 & \ & \ & \ \\[9pt]
a_3 & \ -16a_3 & \ 30a_3 & \ -16a_3 & \ a_3 & \ & \ \\[9pt]
& \ \ddots & \ \ddots & \ \ddots & \ \ddots & \ \ddots & \ \\[9pt]
& \ & \ a_{n-2} & \ -16a_{n-2} & \ 30a_{n-2} & \ -16a_{n-2} & \ a_{n-2} \\[9pt]
& \ & \ & \ a_{n-1} & \ -16a_{n-1} & \ 30a_{n-1} & \ -16a_{n-1} \\[9pt]
& \ & \ & \ & \ & \ -12a_n & 24a_n
\end{bmatrix}\hspace{-1pt}. 
\end{align*}

\paragraph{GLT spectral analysis of the FD discretization matrices}
Let $p(\theta)$ be the trigonometric polynomial associated with the FD formula $\frac1{12^{\vphantom{\mbox{\tiny1}}}}(1,-16,30,-16,1)$ used to approximate the second derivative $-u''(x)$, i.e.,
\[ p(\theta)=\frac1{12}({\rm e}^{-2{\rm i}\theta}-16{\rm e}^{-{\rm i}\theta}+30-16{\rm e}^{{\rm i}\theta}+{\rm e}^{2{\rm i}\theta})=\frac1{12}(30-32\cos\theta+2\cos(2\theta)). \]
Based on Remark~\ref{oss1}, the following result is not unexpected. 
\begin{theorem}\label{FD_T5}
If $a:[0,1]\to\mathbb R$ is continuous and $b,c:[0,1]\to\mathbb R$ are bounded then
\begin{equation}\label{cdr4GLT}
\{P_n\}_n\sim_{\rm GLT}a(x)p(\theta)
\end{equation}
and
\begin{equation}\label{cdr4sigla}
\{P_n\}_n\sim_{\sigma,\lambda}a(x)p(\theta).
\end{equation}
\end{theorem}
\begin{proof}
To simultaneously obtain \eqref{cdr4GLT} and \eqref{cdr4sigla}, we consider the following decomposition of $P_n$:
\begin{equation}\label{Pn_dec}
P_n=\tilde K_n+(K_n-\tilde K_n)+Z_n,
\end{equation}
where $\tilde K_n$ is the symmetric approximation of $K_n$ given by
\begin{align*}
\tilde K_n&=S_n^{\mathcal G_n}(a)\circ T_n(p)\\
&=\frac1{12}\begin{bmatrix}
30 a_1 & \ -16 a_1 & \ a_1 & \ & \ & \ & \ \\[9pt]
-16 a_1 & \ 30 a_2 & \ -16 a_2 & \ a_2 & \ & \ & \ \\[9pt]
 a_1 & \ -16 a_2 & \ 30 a_3 & \ -16 a_3 & \ a_3 & \ & \ \\[9pt]
& \ \ddots & \ \ddots & \ \ddots & \ \ddots & \ \ddots & \ \\[9pt]
& \ & \ a_{n-4} & \ -16 a_{n-3} & \ 30 a_{n-2} & \ -16 a_{n-2} & \ a_{n-2} \\[9pt]
& \ & \ & \ a_{n-3} & \ -16 a_{n-2} & \ 30 a_{n-1} & \ -16 a_{n-1} \\[9pt]
& \ & \ & \ & \ a_{n-2} & \ -16 a_{n-1} & 30 a_n
\end{bmatrix}
\end{align*}
and $\mathcal G_n=\{x_j\}_{j=1,\ldots,n}$.
We show that: 
\begin{enumerate}[nolistsep]
	\item[(a)] $\{\tilde K_n\}_n\sim_{\rm GLT}a(x)p(\theta)$; \vspace{3pt}
	\item[(b)] $\|Z_n\|_2=o(n^{1/2})$; \vspace{3pt}
	\item[(c)] $\|K_n-\tilde K_n\|_2=o(n^{1/2})$.
\end{enumerate}
Once we have proved (a)--(c), the theorem is proved. Indeed, (b)--(c) imply that 
\[ \|(K_n-\tilde K_n)+Z_n\|_2\le\|K_n-\tilde K_n\|_2+\|Z_n\|_2=o(n^{1/2}), \]
which in turn implies that $\{(K_n-\tilde K_n)+Z_n\}_n\sim_\sigma0$ by {\bf Z2} and $\{(K_n-\tilde K_n)+Z_n\}_n\sim_{\rm GLT}0$ by {\bf GLT3}. Hence, the GLT relation \eqref{cdr4GLT} follows from {\bf GLT4}, keeping in mind (a) and the decomposition \eqref{Pn_dec}; the singular value distribution in \eqref{cdr4sigla} follows from \eqref{cdr4GLT} and {\bf GLT1}; and the spectral distribution in \eqref{cdr4sigla} follows from \eqref{cdr4GLT} and {\bf GLT2} applied with $X_n=\tilde K_n$ and $Y_n=(K_n-\tilde K_n)+Z_n$.

\medskip

\noindent{\em Proof of }(a). This follows from Theorem~\ref{Dtilde_lemma}, since $\mathcal G_n$ is a.u.\ in $[0,1]$.

\medskip

\noindent{\em Proof of }(b). We have
\begin{align*}
\|Z_n\|_2\le\left((2n-1)\|b\|_\infty^2h^2\right)^{1/2}+\left((3n-2)\|c\|_\infty^2\frac{h^4}9\right)^{1/2}=(2n-1)^{1/2}\|b\|_\infty h+(3n-2)^{1/2}\|c\|_\infty\frac{h^2}{3}=o(n^{1/2}).
\end{align*}

\medskip

\noindent{\em Proof of }(c). A direct comparison between $K_n$ and $\tilde K_n$ shows that
$$K_n=\tilde K_n+R_n+N_n,$$
where $N_n=K_n-\tilde K_n-R_n$ and $R_n$ is the matrix whose rows are all zeros except for the first and the last one, which are given by
\[ \frac1{12}\left[\,-6a_1\qquad 4a_1\qquad-a_1\qquad0\qquad\cdots\qquad0\,\right] \]
and
\[ \frac1{12}\left[\,0\qquad\cdots\qquad0\qquad -a_{n-2}\qquad -12a_n+16a_{n-1}\qquad -6a_n\,\right], \]
respectively.
We have
\[ \|R_n\|_2^2\le7\|a\|_\infty^2,\qquad\|N_n\|_2^2=\sum_{\ell=1}^{n-2}256(a_{\ell+1}-a_\ell)^2+\sum_{\ell=1}^{n-3}(a_{\ell+2}-a_\ell)^2\le257n\,\omega_a(2h)^2 \]
and
\[ \|K_n-\tilde K_n\|_2\le\|R_n\|_2+\|N_n\|_2\le7^{1/2}\|a\|_\infty+257^{1/2}n^{1/2}\omega_a(2h)=o(n^{1/2}). \tag*{\qedhere} \]
\end{proof}

\begin{remark}[\textbf{nonnegativity and order of the zero at $\boldsymbol{\theta=0}$}]\label{order-zero}
Although we have changed the FD scheme to approximate the second derivative $-u''(x)$, the resulting trigonometric polynomial $p(\theta)$ retains some properties of $2-2\cos\theta$. Indeed, $p(\theta)$ is nonnegative on $[-\pi,\pi]$ and it has a unique zero of order 2 at $\theta=0$, because
$$\lim_{\theta\to0}\frac{p(\theta)}{\theta^2}=1=\lim_{\theta\to0}\frac{2-2\cos\theta}{\theta^2}.$$
This reflects the fact the the associated FD formula $\frac1{12^{\vphantom{1}}}(1,-16,30,-16,1)$ approximates $-u''(x)$, which is a differential operator of order 2 and it is also nonnegative on $\{v\in C^2([0,1]):v(0)=v(1)=0\}$; see also Remark~\ref{oss1.1}.
\end{remark}


\subsubsection{FD Discretization of Higher-Order Equations}\label{ho} 

So far we only considered the FD discretization of second-order DEs. In order to show that the GLT spectral analysis is not limited to second-order DEs, in this section we deal with an higher-order problem. For simplicity, we focus on the following fourth-order problem with homogeneous Dirichlet--Neumann boundary conditions:
\begin{equation}\label{hoPDE}
\left\{\begin{aligned}
&a(x)u^{(4)}(x)=f(x), &&\quad x\in(0,1),\\[3pt]
&u(0)=0,\quad u(1)=0,\\[3pt]
&u'(0)=0,\quad u'(1)=0.
\end{aligned}\right.
\end{equation}
We do not consider more complicated boundary conditions, and we do not include terms with lower-order derivatives, because we know from Remark~\ref{oss1} and the experience gained from Section~\ref{oss2} that both these ingredients only serve to complicate things, but ultimately they do not affect the symbol, as well as the singular value and eigenvalue distribution, of the resulting discretization matrices. Based on Remark~\ref{oss1}, the symbol of the matrix-sequence arising from the FD discretization of \eqref{hoPDE} should be $a(x)q(\theta)$, where $q(\theta)$ is the trigonometric polynomial associated with the FD formula used to discretize $u^{(4)}(x)$. We are going to see that this is in fact the case.

\paragraph{FD discretization}
We approximate the fourth derivative $u^{(4)}(x)$ by the second-order central FD scheme $(1,-4,6,$\linebreak$-4,1)$. 
This scheme yields the approximation
\[ a(x)u^{(4)}(x)|_{x=x_j}\approx a(x_j)\frac{u(x_{j+2})-4u(x_{j+1})+6u(x_j)-4u(x_{j-1})+u(x_{j-2})}{h^4}, \]
for all $j=2,\ldots,n+1$; here, $h=\frac1{n+3^{\vphantom{\mbox{\tiny 1}}}}$ and $x_j=jh$ for $j=0,\ldots,n+3$. Taking into account the homogeneous boundary conditions, we approximate
the nodal value $u(x_j)$ with the value $u_j$
for $j=0,\ldots,n+3$, where $u_0=u_1=u_{n+2}=u_{n+3}=0$ and $\mathbf{u}=(u_2,\ldots,u_{n+1})^T$ is the solution of the linear system
$$a(x_j)(u_{j+2}-4u_{j+1}+6u_j-4u_{j-1}+u_{j-2})=h^4f(x_j),\qquad j=2,\ldots,n+1.$$
The matrix $A_n$ of this linear system is given by
$$A_n=\begin{bmatrix}
6a_2 & -4a_2 & a_2 & & & & \\[7pt]
-4a_3 & 6a_3 & -4a_3 & a_3 & & & \\[7pt]
a_4 & -4a_4 & 6a_4 & -4a_4 & a_4 & & \\[7pt]
& \ddots & \ddots & \ddots & \ddots & \ddots & \\[7pt]
& & a_{n-1} & -4a_{n-1} & 6a_{n-1} & -4a_{n-1} & a_{n-1}\\[7pt]
& & & a_n & -4a_n & 6a_n & -4a_n\\[7pt]
& & & & a_{n+1} & -4a_{n+1} & 6a_{n+1}
\end{bmatrix},$$
where $a_i=a(x_i)$ for all $i=2,\ldots,n+1$.

\paragraph{GLT spectral analysis of the FD discretization matrices}
Let $q(\theta)$ be the trigonometric polynomial associated with the FD formula $(1,-4,6,-4,1)$, i.e.,
\[ q(\theta)={\rm e}^{-2{\rm i}\theta}-4{\rm e}^{-{\rm i}\theta}+6-4{\rm e}^{{\rm i}\theta}+{\rm e}^{2{\rm i}\theta}=6-8\cos\theta+2\cos(2\theta). \] 


\begin{theorem}\label{FD_T6}
If $a:[0,1]\to\mathbb R$ is continuous then
\begin{equation}\label{hoGLT}
\{A_n\}_n\sim_{\rm GLT}a(x)q(\theta)
\end{equation}
and
\begin{equation}\label{hosigla}
\{A_n\}_n\sim_{\sigma,\lambda}a(x)q(\theta).
\end{equation}
\end{theorem}
\begin{proof}
The grid $\mathcal G_n=\{x_{i+1}\}_{i=1,\ldots,n}$ is clearly a.u.\ in $[0,1]$, and it is clear from the expressions of $A_n$ and $q(\theta)$ that
\[ A_n=D_n^{\mathcal G_n}(a)T_n(q). \]
Thus, \eqref{hoGLT} follows immediately from {\bf GLT3}--{\bf GLT4}. The singular value distribution in \eqref{hosigla} follows from \eqref{hoGLT} and {\bf GLT1}. The spectral distribution in \eqref{hosigla} follows from \eqref{hoGLT} and {\bf GLT2} applied to the decomposition
\[ A_n=S_n^{\mathcal G_n}(a)\circ T_n(q)+(A_n-S_n^{\mathcal G_n}(a)\circ T_n(q)), \]
taking into account that $S_n^{\mathcal G_n}(a)\circ T_n(q)$ is symmetric and $\|A_n-S_n^{\mathcal G_n}(a)\circ T_n(q)\|_2=o(n^{1/2})$ by Theorem~\ref{Dtilde_lemma}.
\end{proof}

\begin{remark}[\textbf{nonnegativity and order of the zero at $\boldsymbol{\theta=0}$}]\label{order-zero'}
The polynomial $q(\theta)$ is nonnegative on $[-\pi,\pi]$ and has a unique zero of order 4 at $\theta=0$, because 
$$\lim_{\theta\to0}\frac{q(\theta)}{\theta^4}=1.$$
This reflects the fact that the FD formula $(1,-4,6,-4,1)$ associated with $q(\theta)$ approximates the fourth derivative $u^{(4)}(x)$, which is a differential operator of order 4 (it is also nonnegative on the space of functions $v\in C^4([0,1])$ such that $v(0)=v(1)=0$ and $v'(0)=v'(1)=0$, in the sense that $\int_0^1v^{(4)}(x)v(x){\rm d} x=\int_0^1(v''(x))^2{\rm d} x\ge0$ for all such~$v$); see also Remarks~\ref{oss1.1} and~\ref{order-zero}.
\end{remark}


\subsubsection{Non-Uniform FD Discretizations}\label{non-uniform-FDs}
All the FD discretizations considered in Sections~\ref{1D-case}--\ref{ho} are based on uniform grids. It is natural to ask whether the theory of GLT sequences finds applications also in the context of non-uniform FD discretizations. The answer to this question is affirmative, at least in the case where the non-uniform grid is obtained as the mapping of a uniform grid through a fixed function $G$, independent of the mesh size. In this section we illustrate this claim by means of a simple example.

\paragraph{FD discretization}
Consider the diffusion equation~\eqref{1d.pde}. 
Take a discretization parameter $n\in\mathbb N$, fix a set of grid points $0=x_0<x_1<\ldots<x_{n+1}=1$ and define the corresponding stepsizes $h_j=x_j-x_{j-1}$ for $j=1,\ldots,n+1$. For each $j=1,\ldots,n$, we approximate $-(a(x)u'(x))'|_{x=x_j}$ by the FD formula
{\allowdisplaybreaks\begin{align*}
&-(a(x)u'(x))'|_{x=x_j}\approx-\frac{a(x_j+\frac{h_{j+1}}{2^{\vphantom{1}}})u'(x_j+\frac{h_{j+1}}{2^{\vphantom{1}}})-a(x_j-\frac{h_j}{2^{\vphantom{1}}})u'(x_j-\frac{h_j}{2^{\vphantom{1}}})}{\frac{h_{j+1}}{2^{\vphantom{1}}}+\frac{h_j}{2^{\vphantom{1}}}}\\[5pt]
&\approx-\frac{a(x_j+\frac{h_{j+1}}{2^{\vphantom{1}}})\dfrac{u(x_{j+1})-u(x_j)}{h_{j+1}}-a(x_j-\frac{h_j}{2^{\vphantom{1}}})\dfrac{u(x_j)-u(x_{j-1})}{h_j}}{\frac{h_{j+1}}{2^{\vphantom{1}}}+\frac{h_j}{2^{\vphantom{1}}}}\\[5pt]
&=\frac2{h_j+h_{j+1}}\left[-\frac{a(x_j-\frac{h_j}{2^{\vphantom{1}}})}{h_j}u(x_{j-1})+\biggl(\frac{a(x_j-\frac{h_j}{2^{\vphantom{1}}})}{h_j}
+\frac{a(x_j+\frac{h_{j+1}}{2^{\vphantom{1}}})}{h_{j+1}}\biggr)u(x_j)-\frac{a(x_j+\frac{h_{j+1}}{2^{\vphantom{1}}})}{h_{j+1}}u(x_{j+1})\right].
\end{align*}}%
This means that the 
values of the solution $u$ at the nodes $x_j$, $j=1,\ldots,n$, satisfy (approximately) the following linear system:
\begin{align*}
&-\frac{a(x_j-\frac{h_j}{2^{\vphantom{1}}})}{h_j}u(x_{j-1})+\biggl(\frac{a(x_j-\frac{h_j}{2^{\vphantom{1}}})}{h_j}+\frac{a(x_j+\frac{h_{j+1}}{2^{\vphantom{1}}})}{h_{j+1}}\biggr)u(x_j)-\frac{a(x_j+\frac{h_{j+1}}{2^{\vphantom{1}}})}{h_{j+1}}u(x_{j+1})=\frac{h_j+h_{j+1}}2f(x_j),\\[5pt]
&j=1,\ldots,n.
\end{align*}
We therefore approximate 
the nodal value $u(x_j)$ with the value $u_j$
for $j=0,\ldots,n+1$, where $u_0=\alpha$, $u_{n+1}=\beta$, and $\mathbf u=(u_1,\ldots,u_n)^T$ solves
\begin{align*}
&-\frac{a(x_j-\frac{h_j}{2^{\vphantom{1}}})}{h_j}u_{j-1}+\biggl(\frac{a(x_j-\frac{h_j}{2^{\vphantom{1}}})}{h_j}+\frac{a(x_j+\frac{h_{j+1}}{2^{\vphantom{1}}})}{h_{j+1}}\biggr)u_j-\frac{a(x_j+\frac{h_{j+1}}{2^{\vphantom{1}}})}{h_{j+1}}u_{j+1}=\frac{h_j+h_{j+1}}2f(x_j),\\[5pt]
&j=1,\ldots,n. 
\end{align*}
The matrix of this linear system is the $n\times n$ tridiagonal symmetric matrix given by
\begin{equation}\label{FD-AGn0}
{\rm tridiag}_n\biggl[-\frac{a(x_j-\frac{h_j}{2^{\vphantom{1}}})}{h_j},\ \frac{a(x_j-\frac{h_j}{2^{\vphantom{1}}})}{h_j}+\frac{a(x_j+\frac{h_{j+1}}{2^{\vphantom{1}}})}{h_{j+1}},\ -\frac{a(x_j+\frac{h_{j+1}}{2^{\vphantom{1}}})}{h_{j+1}}\biggr],
\end{equation}
where 
\[ {\rm tridiag}_n[r_j,s_j,t_j]=\begin{bmatrix}
s_1 & t_1 & & & \\
r_2 & s_2 & t_2 & & \\
& \ddots & \ddots & \ddots & \\
& & r_{n-1} & s_{n-1} & t_{n-1}\\
& & & r_n & s_n
\end{bmatrix}. \]

\paragraph{GLT spectral analysis of the FD discretization matrices}
Let $h=\frac1{n+1}$ and $\hat x_j=jh$, $j=0,\ldots,n+1$. In the following, we assume that the set of points $\{x_0,x_1,\ldots,x_{n+1}\}$ is obtained as the mapping of the uniform grid $\{\hat x_0,\hat x_1,\ldots,\hat x_{n+1}\}$ through a fixed function $G$, i.e., $x_j=G(\hat x_j)$ for $j=0,\ldots,n+1$, where $G:[0,1]\to[0,1]$ is an increasing and bijective map, independent of the mesh parameter $n$. The resulting FD discretization matrix \eqref{FD-AGn0} will be denoted by $A_{G,n}$ in order to emphasize its dependence on $G$. In formulas, 
\begin{equation}\label{FD-AGn}
A_{G,n}={\rm tridiag}_n\biggl[-\frac{a(G(\hat x_j)-\frac{h_j}{2^{\vphantom{1}}})}{h_j},\ \frac{a(G(\hat x_j)-\frac{h_j}{2^{\vphantom{1}}})}{h_j}+\frac{a(G(\hat x_j)+\frac{h_{j+1}}{2^{\vphantom{1}}})}{h_{j+1}},\ -\frac{a(G(\hat x_j)+\frac{h_{j+1}}{2^{\vphantom{1}}})}{h_{j+1}}\biggr]
\end{equation}
with
\[ h_j=G(\hat x_j)-G(\hat x_{j-1}),\qquad j=1,\ldots,n+1. \]

\begin{theorem}\label{FD_T7}
Let $a:[0,1]\to\mathbb R$ be continuous. Suppose that $G:[0,1]\to[0,1]$ is an increasing bijective map in $C^1([0,1])$ and there exist at most finitely many points $\hat x$ such that $G'(\hat x)=0$. Then\,\footnote{\,It is understood that the function $\dfrac{a(G(\hat x))}{G'(\hat x)}(2-2\cos\theta)$ is defined a.e.\ by its expression, i.e., for all $(\hat x,\theta)$ such that $G'(\hat x)\ne0$.}
\begin{equation}\label{nuGLT}
\Bigl\{\frac1{n+1}A_{G,n}\Bigr\}_n\sim_{\rm GLT}\frac{a(G(\hat x))}{G'(\hat x)}(2-2\cos\theta)
\end{equation}
and
\begin{equation}\label{nusigla}
\Bigl\{\frac1{n+1}A_{G,n}\Bigr\}_n\sim_{\sigma,\lambda}\frac{a(G(\hat x))}{G'(\hat x)}(2-2\cos\theta).
\end{equation}
\end{theorem}
\begin{proof}
We only prove \eqref{nuGLT} because \eqref{nusigla} follows immediately from \eqref{nuGLT} and {\bf GLT1} as the matrices $A_{G,n}$ are symmetric.
Since $G\in C^1([0,1])$, for every $j=1,\ldots,n$ there exist $\alpha_j\in[\hat x_{j-1},\hat x_j]$ and $\beta_j\in[\hat x_j,\hat x_{j+1}]$ such that
\begin{align}\label{e.1}
h_j&=G(\hat x_j)-G(\hat x_{j-1})=G'(\alpha_j)h=(G'(\hat x_j)+\delta_j)h,\\
h_{j+1}&=G(\hat x_{j+1})-G(\hat x_j)=G'(\beta_j)h=(G'(\hat x_j)+\epsilon_j)h, \label{e.1'}
\end{align}
where 
\begin{align*}
\delta_j &= G'(\alpha_j)-G'(\hat x_j),\\ 
\epsilon_j &= G'(\beta_j)-G'(\hat x_j). 
\end{align*}
Note that
\[ |\delta_j|,|\epsilon_j|\le\omega_{G'}(h),\qquad j=1,\ldots,n. \]
In view of \eqref{e.1}--\eqref{e.1'}, we have, for each $j=1,\ldots,n$,
\begin{align}
a\Bigl(G(\hat x_j)-\frac{h_j}2\Bigr)&=a\Bigl(G(\hat x_j)-\frac h2(G'(\hat x_j)+\delta_j)\Bigr)=a(G(\hat x_j))+\mu_j,\label{e.2}\\[3pt]
a\Bigl(G(\hat x_j)+\frac{h_{j+1}}2\Bigr)&=a\Bigl(G(\hat x_j)+\frac h2(G'(\hat x_j)+\epsilon_j)\Bigr)=a(G(\hat x_j))+\eta_j,\label{e.2'}
\end{align}
where
\begin{align*}
\mu_j&=a\Bigl(G(\hat x_j)-\frac h2(G'(\hat x_j)+\delta_j)\Bigr)-a(G(\hat x_j))=a\Bigl(G(\hat x_j)-\frac h2\,G'(\alpha_j)\Bigr)-a(G(\hat x_j)),\\[3pt]
\eta_j&=a\Bigl(G(\hat x_j)+\frac h2(G'(\hat x_j)+\epsilon_j)\Bigr)-a(G(\hat x_j))=a\Bigl(G(\hat x_j)+\frac h2\,G'(\beta_j)\Bigr)-a(G(\hat x_j)).
\end{align*}
Note that
\[ |\mu_j|,|\eta_j|\le\omega_a\Bigl(\frac h2\|G'\|_\infty\Bigr).
\qquad j=1,\ldots,n, \]
Substituting \eqref{e.1}--\eqref{e.2'} in \eqref{FD-AGn}, we obtain
\begin{align}\label{new_form}
\frac1{n+1}A_{G,n}&=hA_{G,n}\notag\\
&=\textup{tridiag}_n\biggl[-\frac{a(G(\hat x_j))+\mu_j}{G'(\hat x_j)+\delta_j},\; \frac{a(G(\hat x_j))+\mu_j}{G'(\hat x_j)+\delta_j}+\frac{a(G(\hat x_j))+\eta_j}{G'(\hat x_j)+\epsilon_j},\; -\frac{a(G(\hat x_j))+\eta_j}{G'(\hat x_j)+\epsilon_j}\biggr].
\end{align}
Consider the matrix
\begin{equation}\label{diag*toep}
D_n^{\mathcal G_n}\Bigl(\frac{a(G(\hat x))}{G'(\hat x)}\Bigr)T_n(2-2\cos\theta) = \textup{tridiag}_n\biggl[-\frac{a(G(\hat x_j))}{G'(\hat x_j)},\;2\,\frac{a(G(\hat x_j))}{G'(\hat x_j)},\;-\frac{a(G(\hat x_j))}{G'(\hat x_j)}\biggr],
\end{equation}
where $\mathcal G_n=\{\hat x_j\}_{j=1,\ldots,n}$.
Note that this matrix seems to be an ``approximation'' of $\frac1{n+1}A_{G,n}$; cf.\ \eqref{new_form} and \eqref{diag*toep}.
Since the function $a(G(\hat x))/G'(\hat x)$ is continuous a.e.\ in $[0,1]$ and the grid $\mathcal G_n$ is a.u.\ in $[0,1]$, {\bf GLT3} and {\bf GLT4} yield
\[ \biggl\{D_n^{\mathcal G_n}\Bigl(\frac{a(G(\hat x))}{G'(\hat x)}\Bigr)T_n(2-2\cos\theta)\biggr\}_n\sim_{\rm GLT}\frac{a(G(\hat x))}{G'(\hat x)}(2-2\cos\theta). \]
We are going to show that\,\footnote{\,Note that in the left-hand side of \eqref{a.c.s.constant} we have a fixed matrix-sequence and not a sequence of matrix-sequences. This means that the convergence \eqref{a.c.s.constant} is equivalent to
\[ d_{\rm a.c.s.}\biggl(\Bigl\{D_n^{\mathcal G_n}\Bigl(\frac{a(G(\hat x))}{G'(\hat x)}\Bigr)T_n(2-2\cos\theta)\Bigr\}_n,\Bigl\{\frac1{n+1}A_{G,n}\Bigr\}_n\biggr)=0. \] }
\begin{equation}\label{a.c.s.constant}
\Bigl\{D_n^{\mathcal G_n}\Bigl(\frac{a(G(\hat x))}{G'(\hat x)}\Bigr)T_n(2-2\cos\theta)\Bigr\}_n\stackrel{\rm a.c.s.}{\longrightarrow}\Bigl\{\frac1{n+1}A_{G,n}\Bigr\}_n.
\end{equation}
Once this is proved, \eqref{nuGLT} follows immediately from {\bf GLT7}. 

We first prove \eqref{a.c.s.constant} in the case where $G'(\hat x)$ does not vanish over $[0,1]$, so that
\[ m_{G'}=\min_{\hat x\in[0,1]}G'(\hat x)>0 \]
(because $G\in C^1([0,1])$ is increasing). In this case, we show that $\|Z_n\|\to0$, where 
\begin{equation}\label{FD.Z}
Z_n = \frac1{n+1}A_{G,n}-D_n^{\mathcal G_n}\Bigl(\frac{a(G(\hat x))}{G'(\hat x)}\Bigr)T_n(2-2\cos\theta).
\end{equation}
This implies \eqref{a.c.s.constant} by the definition of a.c.s.\ (Definition~\ref{S.a.c.s.}).\,\footnote{\,Take $R_{n,m}=O_n$ (the $n\times n$ zero matrix), $N_{n,m}=Z_n$, $c(m)=0$, $\omega(m)=\frac1m$, and $n_m$ such that $\|Z_n\|\le\frac1m$ for $n\ge n_m$.}
The matrix $Z_n$ is tridiagonal and a straightforward computation based on \eqref{new_form}--\eqref{diag*toep} shows that all its components are bounded in modulus by a quantity that depends only on $n,G,a$ and that converges to 0 as $n\to\infty$. For example, if $j=2,\ldots,n$ then
\begin{align}
|(Z_n)_{j,j-1}|&=\left|\frac{a(G(\hat x_j))+\mu_j}{G'(\hat x_j)+\delta_j}-\frac{a(G(\hat x_j))}{G'(\hat x_j)}\right|\le\left|\frac{a(G(\hat x_j))+\mu_j}{G'(\hat x_j)+\delta_j}-\frac{a(G(\hat x_j))}{G'(\hat x_j)+\delta_j}\right|+\left|\frac{a(G(\hat x_j))}{G'(\hat x_j)+\delta_j}-\frac{a(G(\hat x_j))}{G'(\hat x_j)}\right|\notag\\[3pt]
&=\left|\frac{\mu_j}{G'(\hat x_j)+\delta_j}\right|+\left|\frac{a(G(\hat x_j))\delta_j}{G'(\hat x_j)(G'(\hat x_j)+\delta_j)}\right|
\le\frac{\omega_a(\frac h2\|G'\|_\infty)}{m_{G'}-\omega_{G'}(h)}+\frac{\|a\|_\infty\omega_{G'}(h)}{m_{G'}(m_{G'}-\omega_{G'}(h))}.\label{bound-Z}
\end{align}
Thus, both the $1$-norm and the $\infty$-norm of $Z_n$ are bounded by a quantity that depends only on $n,G,a$ and that converges to 0 as $n\to\infty$, and we conclude that $\|Z_n\|\to0$ as $n\to\infty$ by \eqref{2-norm}.

Now we consider the case where $G$ has a finite number of points $\hat x$ where ${G'(\hat x)=0}$. In this case, the previous argument does not work because $m_{G'}=0$. However, we can still prove \eqref{a.c.s.constant} in the following way. Let $\hat x^{(1)},\ldots,\hat x^{(s)}$ be the points where $G'$ vanishes, and consider the open balls (intervals)
$B(\hat x^{(k)},\frac1m)=\{\hat x\in[0,1]:|\hat x-\hat x^{(k)}|<\frac1m\}$. Since $G\in C^1([0,1])$ is increasing, the function $G'$ is continuous and positive on the compact set given by the complement of
the union $\bigcup_{k=1}^sB(\hat x^{(k)},\frac1m)$, and so
\[ m_{G',m}=\min_{\hat x\in[0,1]\backslash\bigcup_{k=1}^sB(\hat x^{(k)},\frac1m)}G'(\hat x)>0. \]
For all indices $j=1,\ldots,n$ such that $\hat x_j\in[0,1]\backslash\bigcup_{k=1}^sB(\hat x^{(k)},\frac1m)$, the components in the $j$th row of the tridiagonal matrix \eqref{FD.Z} are bounded in modulus by a quantity that depends only on $n,m,G,a$ and that converges to 0 as $n\to\infty$. This becomes immediately clear if we note that, for such indices $j$, the inequality \eqref{bound-Z} holds unchanged with $m_{G'}$ replaced by $m_{G',m}$. The number of remaining rows of $Z_n$ (the rows corresponding to indices $j$ such that $\hat x_j\in\bigcup_{k=1}^sB(\hat x^{(k)},\frac1m)$) is at most $2s(n+1)/m+s$. Indeed, each interval $B(\hat x^{(k)},\frac1m)$ has length $2/m$ (at most) and can contain at most
$2(n+1)/m+1$ grid points $\hat x_j$. Thus, for every $n,m$ we can split the matrix $Z_n$ into the sum of two terms, i.e., 
\[ Z_n = R_{n,m} + N_{n,m}, \] 
where $N_{n,m}$ is obtained from $Z_n$ by setting to zero all the rows corresponding to indices $j$ such that $\hat x_j\in\bigcup_{k=1}^sB(\hat x^{(k)},\frac1m)$ and $R_{n,m}=Z_n-N_{n,m}$ is obtained from $Z_n$ by setting to zero all the rows corresponding to indices $j$ such that $\hat x_j\in[0,1]\backslash\bigcup_{k=1}^sB(\hat x^{(k)},\frac1m)$. From the above discussion and \eqref{2-norm}, we have
\[ \lim_{n\to\infty}\|N_{n,m}\|\le\lim_{n\to\infty}\sqrt{|N_{n,m}|_1|N_{n,m}|_\infty}=0 \]
for all $m$, and
\[ {\rm rank}(R_{n,m})\le\frac{2s(n+1)}{m}+s \]
for all $m,n$. In particular, for each $m$ we can choose $n_m$ such that, for $n\ge n_m$, ${\rm rank}(R_{n,m})\le3sn/m$ and $\|N_{n,m}\|\le1/m$. The convergence \eqref{a.c.s.constant} now follows from the definition of a.c.s.\ (Definition~\ref{S.a.c.s.}).
\end{proof}

An increasing bijective map $G:[0,1]\to[0,1]$ in $C^1([0,1])$ is said to be regular if $G'(\hat x)\ne0$ for all $\hat x\in[0,1]$ and is said to be singular otherwise, i.e., if $G'(\hat x)=0$ for some $\hat x\in[0,1]$. If $G$ is singular, any point $\hat x\in[0,1]$ such that $G'(\hat x)=0$ is referred to as a singularity point (or simply a singularity) of $G$. 
The choice of a map $G$ with one or more singularity points corresponds to adopting a local refinement strategy, according to which the grid points $x_j$ rapidly accumulate at the $G$-images of the singularities as $n$ increases. For example, if
\begin{equation}\label{Gq}
G(\hat x)=\hat x^q, \qquad q>1,
\end{equation}
then 0 is a singularity of $G$ (because $G'(0)=0$) and the grid points
\[ x_j=G(\hat x_j)=\Bigl(\frac{j}{n+1}\Bigr)^q,\quad\ j=0,\ldots,n+1,\]
rapidly accumulate at $G(0)=0$ as $n\to\infty$. We note that, whenever $G$ is singular, the symbol in \eqref{nuGLT} is unbounded (except in some rare cases where $a(G(\hat x))$ and $G'(\hat x)$ vanish simultaneously).

\subsection{FE Discretization of DEs}\label{sec:FEs}

\subsubsection{FE Discretization of Convection-Diffusion-Reaction Equations}\label{FEs}

Consider the following convection-diffusion-reaction problem in divergence form with Dirichlet boundary conditions: 
\begin{equation}\label{toy_ex1.G}
\left\{\begin{aligned}
&-(a(x)u'(x))'+b(x)u'(x)+c(x)u(x) = f(x), &&\quad x\in(0,1),\\[3pt]
&u(0)=u(1)=0.
\end{aligned}\right.
\end{equation}

\paragraph{Weak form}
Before considering the FE discretization of \eqref{toy_ex1.G}, we need to introduce the weak form of \eqref{toy_ex1.G}.
We first recall that, if $\Omega\subset\mathbb R$ is a bounded interval whose endpoints are, say, $\alpha$ and $\beta$, the Sobolev spaces $H^1(\Omega)$ and $H^1_0(\Omega)$ are defined as follows \cite[Chapter~8]{Brezis}:
\begin{align*}
H^1(\Omega)&=\biggl\{v:\overline\Omega\to\mathbb R:v\in C(\overline\Omega),\ v\mbox{ is differentiable a.e.\ with }v'\in L^2(\Omega),\\
&\hphantom{{}=\biggl\{v:\overline\Omega\to\mathbb R:{}}v(x)=v(\alpha)+\int_\alpha^xv'(y){\rm d}y\,\mbox{ for all }\,x\in\overline\Omega\biggr\},\\[8pt]
H^1_0(\Omega)&=\{v\in H^1(\Omega):v(\alpha)=v(\beta)=0\}.
\end{align*}
The derivative $v'$ of a function $v\in H^1(\Omega)$, which exists a.e.\ in $\Omega$ but may not be defined at some points of $\Omega$, is called weak (Sobolev) derivative.

\begin{example}\label{basic_exe}
We show that the hat-function\,\footnote{\,The name ``hat-function'' is due to the shape of the graph of $v$, which looks like a pointed hat.}
\[ v(x)=2x\chi_{\left[0,\frac1{2}\right)}(x)+(2-2x)\chi_{\left[\frac1{2},1\right)}(x)=\left\{\begin{aligned}&2x, &\quad x\in\bigl[0,\textstyle{\frac12}\bigr),\\&2-2x, &\quad x\in\bigl[\textstyle{\frac12},1\bigr],\end{aligned}\right. \]
belongs to $H^1_0([0,1])$. It is clear that $v\in C([0,1])$, $v(0)=v(1)=0$, and $v$ is differentiable a.e.\ in $[0,1]$ with
\[ v'(x)=\left\{\begin{aligned}&2, &\quad x\in\bigl[0,\textstyle{\frac12}\bigr),\\&-2, &\quad x\in\bigl(\textstyle{\frac12},1\bigr],\end{aligned}\right. \]
belonging to $L^2([0,1])$. It only remains to prove that $v$ is related to its derivative $v'$ through the fundamental formula of integral calculus, i.e.,
\[ v(x)=v(0)+\int_0^xv'(y){\rm d}y,\qquad\forall\,x\in[0,1]. \]
This is a matter of computations: taking into account that $v(0)=0$, for every $x\in[0,1]$ we have
\begin{align*}
v(0)+\int_0^xv'(y){\rm d}y&=\int_0^x\Bigl(2\chi_{\left[0,\frac1{2}\right)}(y)-2\chi_{\left[\frac1{2},1\right]}(y)\Bigr){\rm d}y=\left\{\begin{aligned}
&\int_0^x2\hspace{0.5pt}{\rm d}y=2x, &\quad x\in\bigl[0,\textstyle{\frac12}\bigr),\\
&1+\int_{\frac1{2}}^x-2\hspace{0.5pt}{\rm d}y=2-2x, &\quad x\in\bigl[\textstyle{\frac12},1\bigr],
\end{aligned}\right.\\
&=v(x).
\end{align*}
\end{example}


Multiplying both sides of the DE in \eqref{toy_ex1.G} by a generic test function $w\in H^1_0([0,1])$ and integrating (formally) over $[0,1]$, we obtain the {\em weak form} of \eqref{toy_ex1.G}, which
reads as follows \cite[Chapter~8]{Brezis}: find $u\in H^1_0([0,1])$ such that 
\begin{equation}\label{toy-weak.G}
\textup{a}(u,w)=\textup{f}(w),\qquad\forall\,w\in H^1_0([0,1]),
\end{equation}
where
\begin{align*}
\textup{a}(u,w)&=\int_0^1a(x)u'(x)w'(x){\rm d} x+\int_0^1b(x)u'(x)w(x){\rm d} x+\int_0^1c(x)u(x)w(x){\rm d} x,\\
\textup{f}(w)&=\int_0^1f(x)w(x){\rm d} x.
\end{align*}
The solution $u$ of \eqref{toy_ex1.G} is also a solution of \eqref{toy-weak.G}, which means that the set of solutions of \eqref{toy-weak.G} contains the solution $u$ of \eqref{toy_ex1.G}. In particular, if the solution of \eqref{toy-weak.G} is unique, it coincides with the solution $u$ of \eqref{toy_ex1.G}.
If the solution $u$ of \eqref{toy_ex1.G} does not exist---which may happen, for instance, if the coefficient $a$ is not differentiable on $(0,1)$---but the solution $u$ of \eqref{toy-weak.G} exists and is unique, then the latter is taken, by definition, as the solution of \eqref{toy_ex1.G}.
In the FE method, we look for an approximation of the solution $u$ of \eqref{toy-weak.G}, assuming it is unique.
Existence and uniqueness of the solution $u$ of \eqref{toy-weak.G} are guaranteed under certain assumptions by the Lax--Milgram theorem; see \cite[Chapter~8]{Brezis}.

\paragraph{FE discretization}
We consider the approximation of \eqref{toy_ex1.G} by classical linear FEs on a uniform mesh in $[0,1]$ with stepsize $h=\frac1{n+1}$. We briefly describe here this approximation technique and for more details we refer the reader to \cite[Chapter~4]{Q} or to any other good book on FEs.
Let $n\in\mathbb N$, $h=\frac1{n+1}$ and $x_i=ih$, $i=0,\ldots,n+1$. Consider the subspace $\mathscr W_n=\textup{span}(\varphi_1,\ldots,\varphi_n)\subset H^1_0([0,1])$, where $\varphi_1,\ldots,\varphi_n$ are the so-called hat-functions:
\begin{align}\label{hat_funs.G}
\varphi_i(x)&=\frac{x-x_{i-1}}{x_i-x_{i-1}}\chi_{[x_{i-1},\,x_i)}(x)+\frac{x_{i+1}-x}{x_{i+1}-x_i}\chi_{[x_i,\,x_{i+1})}(x),\qquad i=1,\ldots,n;
\end{align}
see Figure~\ref{hat-funs}. 
It is not difficult to see that $\mathscr W_n$ is the space of continuous piecewise linear functions corresponding to the sequence of points $0=x_0<x_1<\ldots<x_{n+1}=1$ and vanishing on the boundary of $[0,1]$, i.e.,
\begin{align*}
\mathscr W_n&=\bigl\{s:[0,1]\to\mathbb R:s\in C([0,1]),\\
&\hphantom{{}=\bigl\{s:[0,1]\to\mathbb R:{}}s_{\left|\left[\frac i{n+1^{\vphantom{1}}},\frac{i+1}{n+1^{\vphantom{1}}}\right)\right.}\in\mathbb P_1\,\mbox{ for }\,i=0,\ldots,n,\ s(0)=s(1)=0\bigr\},
\end{align*}
where $\mathbb P_1$ is the space of polynomials of degree less than or equal to~1.
\begin{figure}[t]
\centering
\includegraphics[width=0.60\textwidth]{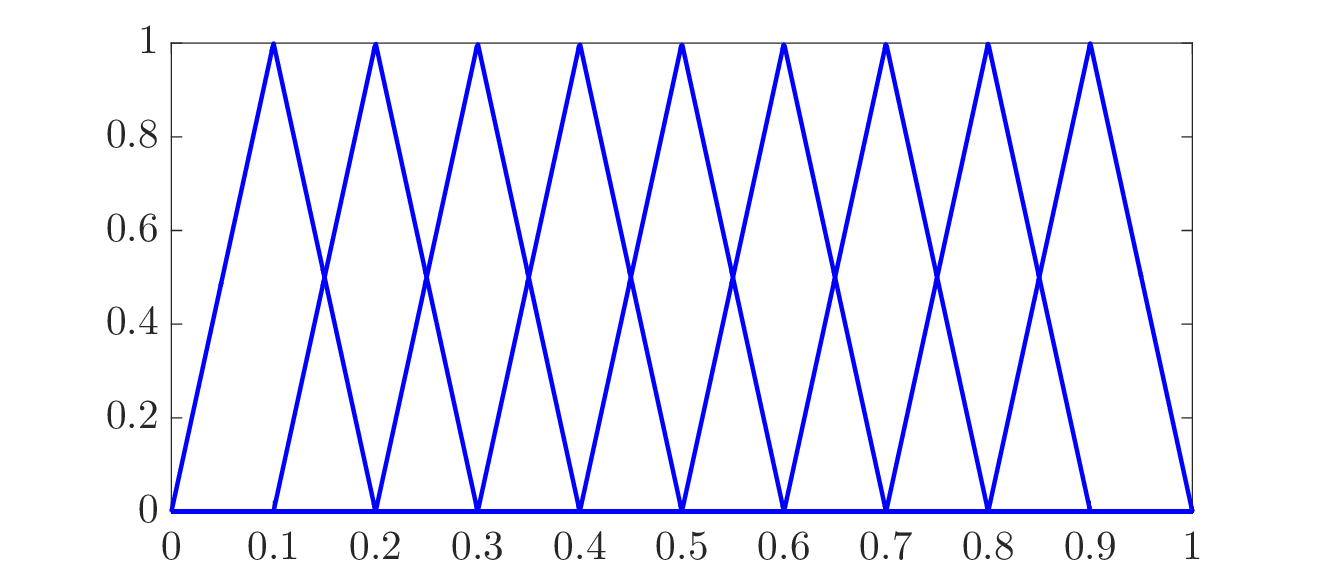}
\caption{Graph of the hat-functions $\varphi_1,\ldots,\varphi_n$ for $n=9$.}
\label{hat-funs}
\end{figure}
In the linear FE approach based on the uniform mesh $\{x_0,\ldots,x_{n+1}\}$, we look for an approximation $u_{\mathscr W_n}$ of $u$ by solving the following (Galerkin) problem: find $u_{\mathscr W_n}\in\mathscr W_n$ such that 
\begin{equation}\label{toy-weak-g.G}
\textup{a}(u_{\mathscr W_n},w)=\textup{f}(w),\qquad\forall\,w\in\mathscr W_n.
\end{equation}
Note that problem \eqref{toy-weak-g.G} is the same as problem \eqref{toy-weak.G}, with the only difference that $H^1_0([0,1])$ is replaced by the subspace $\mathscr W_n$. When $n\to\infty$, the subspace $\mathscr W_n$ tends to ``invade'' the whole space $H^1_0([0,1])$.\,\footnote{\,This can be seen by observing that every $w\in H^1_0([0,1])$ is the limit, say, in $\infty$-norm, of a sequence of functions $w_n\in\mathscr W_n$ (take $w_n$ as the piecewise linear interpolant of $w$ on the nodes $x_0,\ldots,x_{n+1}$).} We can therefore expect that the solution $u_{\mathscr W_n}$ of \eqref{toy-weak-g.G} will converge (in some topology) to the solution $u$ of \eqref{toy-weak.G} as $n\to\infty$, which is in fact the case \cite[Chapter~4]{Q}.
Now, since $\{\varphi_1,\ldots,\varphi_n\}$ is a basis of $\mathscr W_n$, we can write $u_{\mathscr W_n}=\sum_{j=1}^nu_j\varphi_j$ for a unique vector $\mathbf{u}=(u_1,\ldots,u_n)^T$. Moreover, by linearity, the equation in \eqref{toy-weak-g.G} is satisfied for all $w\in\mathscr W_n$ if and only if it is satisfied for the basis functions $\varphi_1,\ldots,\varphi_n$, i.e., if and only if
\begin{align*}
\textup{a}(u_{\mathscr W_n},\varphi_i)=\textup{f}(\varphi_i)\,\mbox{ for all }\,i=1,\ldots,n&\ \iff\ \textup{a}\Biggl(\hspace{0.75pt}\sum_{j=1}^nu_j\varphi_j,\varphi_i\Biggr)=\textup{f}(\varphi_i)\,\mbox{ for all }\,i=1,\ldots,n\\
&\ \iff\ \sum_{j=1}^nu_j\hspace{0.5pt}\textup{a}(\varphi_j,\varphi_i)=\textup{f}(\varphi_i)\,\mbox{ for all }\,i=1,\ldots,n.
\end{align*}
Thus, the computation of $u_{\mathscr W_n}$ (i.e., of $\mathbf{u}$) reduces to solving the linear system
\begin{equation*} 
A_n\mathbf{u}=\mathbf{f},
\end{equation*}
where $\mathbf{f}=(\textup{f}(\varphi_1),\ldots,\textup{f}(\varphi_n))^T$ and $A_n$ is the stiffness matrix,
\begin{equation*}
A_n=[\textup{a}(\varphi_j,\varphi_i)]_{i,j=1}^n. 
\end{equation*}
Note that $A_n$ admits the following decomposition: 
\begin{equation}\label{A.dec.G}
A_n=K_n+Z_n,
\end{equation}
where
\begin{equation}\label{Kn(a).G}
K_n=\left[\int_0^1a(x)\varphi_j'(x)\varphi_i'(x){\rm d} x\right]_{i,j=1}^n
\end{equation}
is the (symmetric) diffusion matrix and
\begin{equation}\label{Zn.G}
Z_n=\left[\int_0^1b(x)\varphi_j'(x)\varphi_i(x){\rm d} x\right]_{i,j=1}^n+\left[\int_0^1c(x)\varphi_j(x)\varphi_i(x){\rm d} x\right]_{i,j=1}^n
\end{equation}
is the sum of the convection and reaction matrix.

\paragraph{GLT spectral analysis of the FE discretization matrices}
Using the theory of GLT sequences we now derive the spectral and singular value distribution of the sequence of normalized stiffness matrices $\{\frac1{n+1}A_n\}_n$ under very weak hypotheses on the DE coefficients $a,b,c$. Throughout this work, if $f:D\subseteq\mathbb R^k\to\mathbb C$ is a function in $L^1(D)$,
we set $\|f\|_{L^1}=\int_D|f(\xx)|{\rm d}\xx$. 

\begin{theorem}\label{FE_T1}
If $a,b,c:[0,1]\to\mathbb R$ belong to $L^1([0,1])$ then
\begin{equation}\label{FEcdrGLT}
\Bigl\{\frac1{n+1}A_n\Bigr\}_n\sim_{\rm GLT}a(x)(2-2\cos\theta)
\end{equation}
and
\begin{equation}\label{FEcdrsigla}
\Bigl\{\frac1{n+1}A_n\Bigr\}_n\sim_{\sigma,\lambda}a(x)(2-2\cos\theta).
\end{equation}
\end{theorem}
\begin{proof}
The proof consists of the following steps.

\medskip

\noindent{\em Step~1.} We show that 
\begin{equation}\label{Zn->0}
\biggl\|\frac1{n+1}Z_n\biggr\|_2=o(n^{1/2}).
\end{equation}
We first note that $Z_n$ is a banded (tridiagonal) matrix, because $(Z_n)_{ij}=0$ for all $i,j=1,\ldots,n$ such that $|i-j|>1$, due to the local support property of the hat-functions: $\textup{supp}(\varphi_i)=[x_{i-1},x_{i+1}]$, $i=1,\ldots,n$. Moreover, since $|\varphi_i(x)|\le1$ for every $x\in[0,1]$ and $|\varphi_i'(x)|\le n+1$ for almost every $x\in[0,1]$, for all $i,j=1,\ldots,n$ we have
\begin{align*}
|(Z_n)_{ij}|&=\left|\int_0^1b(x)\varphi_j'(x)\varphi_i(x){\rm d} x+\int_0^1c(x)\varphi_j(x)\varphi_i(x){\rm d} x\right|\\
&=\left|\int_{x_{i-1}}^{x_{i+1}}b(x)\varphi_j'(x)\varphi_i(x){\rm d} x+\int_{x_{i-1}}^{x_{i+1}}c(x)\varphi_j(x)\varphi_i(x){\rm d} x\right|\\
&\le(n+1)\int_{x_{i-1}}^{x_{i+1}}|b(x)|{\rm d}x+\int_{x_{i-1}}^{x_{i+1}}|c(x)|{\rm d} x\le(n+1)\,\omega_b^{\rm int}\Bigl(\frac2{n+1}\Bigr)+\|c\|_{L^1}.
\end{align*}
Thus,
\[ \|Z_n\|_2^2\le3n\left[(n+1)\,\omega_b^{\rm int}\Bigl(\frac2{n+1}\Bigr)+\|c\|_{L^1}\right]^2, \]
and \eqref{Zn->0} is proved.

\medskip

\noindent{\em Step~2.} Consider the linear operator $K_n(\cdot):L^1([0,1])\to\mathbb C^{n\times n}$,
\begin{equation*} 
K_n(g)=\left[\int_0^1g(x)\varphi_j'(x)\varphi_i'(x){\rm d} x\right]_{i,j=1}^n.
\end{equation*}
By \eqref{Kn(a).G}, we have $K_n=K_n(a)$. The next three steps are devoted to show that
\begin{equation}\label{L->K}
\Bigl\{\frac1{n+1}K_n(g)\Bigr\}_n\sim_{\rm GLT}g(x)(2-2\cos\theta),\qquad\forall\,g\in L^1([0,1]).
\end{equation}
Once this is done, the theorem is proved. Indeed, by applying \eqref{L->K} with ${g=a}$ we immediately get $\{\frac1{n+1}K_n\}_n\sim_{\rm GLT}a(x)(2-2\cos\theta)$. Since $\{\frac1{n+1}Z_n\}_n\sim_\sigma0$ by Step~1 and {\bf Z2}, the GLT relation \eqref{FEcdrGLT} follows from the decomposition
\begin{equation}\label{A.dec.norm.G}
\frac1{n+1}A_n=\frac1{n+1}K_n+\frac1{n+1}Z_n
\end{equation}
and {\bf GLT3}--{\bf GLT4}; the singular value distribution in \eqref{FEcdrsigla} follows from \eqref{FEcdrGLT} and {\bf GLT1}; and the spectral distribution in \eqref{FEcdrsigla} follows from \eqref{FEcdrGLT} and {\bf GLT2} applied to the decomposition \eqref{A.dec.norm.G}, taking into account Step~1. 

\medskip

\noindent{\em Step~3.}
We first prove \eqref{L->K} in the constant-coefficient case where $g=1$ identically. In this case, a direct computation based on \eqref{hat_funs.G} shows that
\begin{equation*} 
K_n(1)=\left[\int_0^1\varphi_j'(x)\varphi_i'(x){\rm d} x\right]_{i,j=1}^n=\frac{1}h\begin{bmatrix}
2 & -1 & \\
-1 & 2 & -1 & \\
& \ddots & \ddots & \ddots \\
& & -1 & 2  & -1\\
& & & -1 & 2
\end{bmatrix}=\frac1h\,T_n(2-2\cos\theta),
\end{equation*}
and the desired relation $\{\frac1{n+1}K_n(1)\}_n\sim_{\rm GLT}2-2\cos\theta$ follows from {\bf GLT3}. 
Two remarks are in order before proceeding with the next step.
\begin{itemize}[nolistsep,leftmargin=*]
	\item It is precisely the analysis of the constant-coefficient case considered in this step that allows one to realize what is the correct normalization factor. In our case, this is $\frac1{n+1}$, which removes the factor $\frac1h$ from $K_n(1)$ and yields a normalized matrix $\frac1{n+1}K_n(1)=T_n(2-2\cos\theta)$, whose components are {\em bounded away from $0$ and $\infty$} (in the present case, they are even constant).
	\item The analysis of the constant-coefficient case considered in this step was not difficult because we are dealing with simple linear FEs. In general, however, the constant-coefficient case could be very difficult to handle. In particular, an explicit construction of the matrix $K_n(1)$ as in the present case is impossible in general, and therefore one must be content to obtain an approximate construction (say, up to a small-rank plus a small-norm correction) and infer the corresponding GLT relation therefrom. This is what happens, for instance, in the case of higher-order FEs, where the basis functions $\varphi_1,\ldots,\varphi_n$ are replaced by higher-degree B-splines; see \cite[Step~3 of Theorem~10.15]{GLT-bookI}.
\end{itemize}

\medskip

\noindent{\em Step~4.}
Now we prove \eqref{L->K} in the case where $g\in C([0,1])$. We first illustrate the idea, and then we go into the details. The proof is based on the fact that the hat-functions \eqref{hat_funs.G} are ``locally supported''. Indeed, the support $[x_{i-1},x_{i+1}]$ of the $i$th hat-function $\varphi_i(x)$ is located near the point $\frac in\in[x_i,x_{i+1}]$, and the width of the support tends to 0 as $n\to\infty$. In this sense, the linear FE method considered herein belongs to the family of the so-called ``local'' methods. Since $g(x)$ varies continuously over $[0,1]$, the $(i,j)$ entry of $K_n(g)$ can be approximated as follows, for every $i,j=1,\ldots,n$:
\begin{align*} 
(K_n(g))_{ij}&=\int_0^1g(x)\varphi_j'(x)\varphi_i'(x){\rm d} x=\int_{x_{i-1}}^{x_{i+1}}g(x)\varphi_j'(x)\varphi_i'(x){\rm d} x\\
&\approx g\Bigl(\frac in\Bigr)\int_{x_{i-1}}^{x_{i+1}}\varphi_j'(x)\varphi_i'(x){\rm d} x=g\Bigl(\frac in\Bigr)\int_0^1\varphi_j'(x)\varphi_i'(x){\rm d} x=g\Bigl(\frac in\Bigr)(K_n(1))_{ij}.
\end{align*}
This approximation can be rewritten in matrix form as
\begin{equation}\label{K-approx}
K_n(g)\approx D_n(g)K_n(1).
\end{equation}
We will see that \eqref{K-approx} implies that $\{\frac1{n+1}K_n(g)-\frac1{n+1}D_n(g)K_n(1)\}_n\sim_\sigma0$, and \eqref{L->K} will then follow from Step~3 and {\bf GLT3}--{\bf GLT4}, taking into account the obvious decomposition
\begin{equation}\label{odec}
\frac1{n+1}K_n(g)=\frac1{n+1}D_n(g)K_n(1)+Y_n,\qquad Y_n=\frac1{n+1}K_n(g)-\frac1{n+1}D_n(g)K_n(1).
\end{equation}
Let us now go into the details. Since ${\rm supp}(\varphi_i)=[x_{i-1},x_{i+1}]$ and $|\varphi_i'(x)|\le n+1$ for almost every $x\in[0,1]$, for all $i,j=1,\ldots,n$ we have
\begin{align*}
\left|(K_n(g))_{ij}-(D_n(g)K_n(1))_{ij}\right|&=\left|\int_0^1\Bigl[g(x)-g\Bigl(\frac in\Bigr)\Bigr]\varphi_j'(x)\varphi_i'(x){\rm d} x\right|\le(n+1)^2\int_{x_{i-1}}^{x_{i+1}}\Bigl|g(x)-g\Bigl(\frac in\Bigr)\Bigr|{\rm d} x\\
&\le(n+1)^2\int_{x_{i-1}}^{x_{i+1}}\omega_g\Bigl(\frac2{n+1}\Bigr){\rm d}x=2(n+1)\,\omega_g\Bigl(\frac2{n+1}\Bigr).
\end{align*}
It follows that each entry of the matrix $Y_n=\frac1{n+1}K_n(g)-\frac1{n+1}D_n(g)K_n(1)$ is bounded in modulus by $2\,\omega_g(\frac2{n+1})$. Moreover, $Y_n$ is banded (tridiagonal), because of the local support property of the hat-functions. 
Thus, both the 1-norm and the $\infty$-norm of $Y_n$ are bounded by $6\hspace{0.75pt}\omega_g(\frac2{n+1})$, and \eqref{2-norm} yields $\|Y_n\|\le6\hspace{0.75pt}\omega_g(\frac2{n+1})\to0$ as $n\to\infty$. Hence, $\{Y_n\}_n\sim_\sigma0$ by {\bf Z1} or {\bf Z2}, which implies \eqref{L->K} by Step~3 and {\bf GLT3}--{\bf GLT4}, in view of the obvious decomposition~\eqref{odec}.

\medskip

\noindent{\em Step~5.}
Finally, we prove \eqref{L->K} in the general case where $g\in L^1([0,1])$. By the density of $C([0,1])$ in $L^1([0,1])$---see \cite[Theorem~3.14]{Rudinone}---there exist continuous functions $g_m\in C([0,1])$ such that $g_m\to g$ in $L^1([0,1])$. By Step~4, 
\begin{equation}\label{GLT7a}
\Bigl\{\frac1{n+1}K_n(g_m)\Bigr\}_n\sim_{\rm GLT}g_m(x)(2-2\cos\theta).
\end{equation}
Moreover, 
\begin{equation}\label{GLT7b}
g_m(x)(2-2\cos\theta)\to g(x)(2-2\cos\theta)\mbox{ \ in measure}.
\end{equation}
We show that 
\begin{equation}\label{GLT7c}
\Bigl\{\frac1{n+1}K_n(g_m)\Bigr\}_n\stackrel{\rm a.c.s.}{\longrightarrow}\Bigl\{\frac1{n+1}K_n(g)\Bigr\}_n,
\end{equation}
after which the desired relation \eqref{L->K} follows from \eqref{GLT7a}--\eqref{GLT7c} and {\bf GLT7}.
Since we have $\sum_{i=1}^n|\varphi_i'(x)|\le2(n+1)$ for almost every $x\in[0,1]$, by \eqref{trace-norm-Lusin} we obtain
\begin{align*}
\|K_n(g)-K_n(g_m)\|_1&\le\sum_{i,j=1}^n|(K_n(g))_{ij}-(K_n(g_m))_{ij}|=\sum_{i,j=1}^n\left|\int_0^1\bigl[g(x)-g_m(x)\bigr]\varphi_j'(x)\varphi_i'(x){\rm d} x\right|\\
&\le\int_0^1\bigl|g(x)-g_m(x)\bigr|\sum_{i,j=1}^n|\varphi_j'(x)|\,|\varphi_i'(x)|{\rm d} x\le4(n+1)^2\|g-g_m\|_{L^1}
\end{align*}
and
\[ \Bigl\|\frac1{n+1}K_n(g)-\frac1{n+1}K_n(g_m)\Bigr\|_1\le5n\|g-g_m\|_{L^1}.
\]
Thus, $\{\frac1{n+1}K_n(g_m)\}_n\stackrel{\rm a.c.s.}{\longrightarrow}\{\frac1{n+1}K_n(g)\}_n$ by {\bf ACS1}.
\end{proof}


\begin{remark}[\textbf{formal structure of the symbol}]\label{oss1.FE}
Problem \eqref{toy_ex1.G} can be formally rewritten as follows:
\begin{equation}\label{problem.FE}
\left\{\begin{aligned}
&-a(x)u''(x)+(b(x)-a'(x))u'(x)+c(x)u(x)=f(x), &&\quad x\in(0,1),\\[3pt]
&u(0)=u(1)=0.
\end{aligned}\right.
\end{equation}
It is therefore clear that the symbol $a(x)(2-2\cos\theta)$ has the same formal structure as the higher-order differential operator $-a(x)u''(x)$ associated with \eqref{problem.FE} (as in the FD case; see Remark~\ref{oss1}). The formal analogy becomes even more evident if we note that $2-2\cos\theta$ is the trigonometric polynomial in the Fourier variable coming from the FE discretization of the (negative) second derivative $-u''(x)$. Indeed, as we have seen in Step~3 of the proof of Theorem~\ref{FE_T1}, ${2-2\cos\theta}$ is the symbol of the sequence of FE diffusion matrices $\{\frac1{n+1}K_n(1)\}_n$, which arises from the FE approximation of the Poisson problem
\begin{equation*}
\left\{\begin{aligned}
&-u''(x)=f(x), &&\quad x\in(0,1),\\[3pt]
&u(0)=u(1)=0,
\end{aligned}\right.
\end{equation*}
i.e., problem~\eqref{toy_ex1.G} in the case where $a(x)=1$ and $b(x)=c(x)=0$ identically.
\end{remark}

\begin{remark}\label{mass_matrix}
Consider the linear operator $M_n(\cdot):L^1([0,1])\to\mathbb C^{n\times n}$,
\[ M_n(g)=\left[\int_0^1g(x)\varphi_j(x)\varphi_i(x){\rm d}x\right]_{i,j=1}^n. \]
A simple adaptation of the arguments used in Steps~3--5 of the proof of Theorem~\ref{FE_T1} allows one to show that
\begin{equation}\label{Mn(g)}
\{(n+1)M_n(g)\}_n\sim_{\rm GLT}g(x)\Bigl(\frac23+\frac13\cos\theta\Bigr),\qquad\forall\,g\in L^1([0,1]).
\end{equation}
Note that the reaction matrix appearing in the right-hand side of \eqref{Zn.G} coincides with
\begin{equation*}
M_n(c)=\left[\int_0^1c(x)\varphi_j(x)\varphi_i(x){\rm d}x\right]_{i,j=1}^n.
\end{equation*}
\end{remark}


\subsubsection{FE Discretization of a System of Equations}\label{le}

In this section we consider the linear FE approximation of a system of DEs, namely
\begin{equation}\label{toy_ex1}
\left\{
\begin{aligned}
-(a(x)u'(x))'+v'(x)&=f(x), &\quad x\in(0,1),\\[3pt]
-u'(x)-\rho v(x)&=g(x), &\quad x\in(0,1),\\[3pt]
u(0)=0,\quad u(1)=0,& \\[3pt]
v(0)=0,\quad v(1)=0.&
\end{aligned}\right.
\end{equation}
As we shall see, the resulting discretization matrices appear in the so-called saddle point form \cite[p.~3]{BGL}, and we will illustrate the way to compute the spectral and singular value distribution of their Schur complements using the theory of GLT sequences. It is worth noting that the Schur complement is a key tool for the numerical treatment of the related linear systems \cite[Section~5]{BGL}. 

\paragraph{FE discretization}
We consider the approximation of \eqref{toy_ex1} by linear FEs on a uniform mesh in $[0,1]$ with stepsize $h=\frac1{n+1}$. Let us describe it shortly.
Multiplying both sides of the DEs in \eqref{toy_ex1} by a generic test function $w\in H^1_0([0,1])$ and integrating (formally) over $[0,1]$, we obtain the weak form of \eqref{toy_ex1}, which reads as follows:
find $u,v\in H^1_0([0,1])$ such that 
\begin{equation}\label{toy-weak}
\left\{\begin{aligned}
\textstyle{\int_0^1a(x)u'(x)w'(x){\rm d} x + \int_0^1v'(x)w(x){\rm d} x} &= \textstyle{\int_0^1f(x)w(x){\rm d}x,\qquad\forall\,w\in H^1_0([0,1]),}\\[3pt]
\textstyle{-\int_0^1u'(x)w(x){\rm d} x - \rho\int_0^1v(x)w(x){\rm d} x} &= \textstyle{\int_0^1g(x)w(x){\rm d}x,\qquad\forall\,w\in H^1_0([0,1]).}
\end{aligned}\right.
\end{equation}
Let $n\in\mathbb N$, $h=\frac1{n+1}$ and $x_i=ih$, $i=0,\ldots,n+1$. In the linear FE approach based on the mesh $\{x_0,\ldots,x_{n+1}\}$, we fix the subspace $\mathscr W_n=\textup{span}(\varphi_1,\ldots,\varphi_n)\subset H^1_0([0,1])$, where $\varphi_1,\ldots,\varphi_n$ are the hat-functions in \eqref{hat_funs.G} (see also Figure~\ref{hat-funs}). Then, we look for approximations $u_{\mathscr W_n},v_{\mathscr W_n}$ of $u,v$ by solving the following (Galerkin) problem: find $u_{\mathscr W_n},v_{\mathscr W_n}\in\mathscr W_n$ such that 
\begin{equation}\label{toy-weak-g}
\left\{\begin{aligned}
\textstyle{\int_0^1a(x)u_{\mathscr W_n}'(x)w'(x){\rm d} x + \int_0^1v_{\mathscr W_n}'(x)w(x){\rm d} x} &= \textstyle{\int_0^1f(x)w(x){\rm d}x,\qquad\forall\,w\in\mathscr W_n,}\\[3pt]
\textstyle{-\int_0^1u_{\mathscr W_n}'(x)w(x){\rm d} x - \rho\int_0^1v_{\mathscr W_n}(x)w(x){\rm d} x} &= \textstyle{\int_0^1g(x)w(x){\rm d}x,\qquad\forall\,w\in\mathscr W_n.}
\end{aligned}\right.
\end{equation}
Since $\{\varphi_1,\ldots,\varphi_n\}$ is a basis of $\mathscr W_n$, we can write $u_{\mathscr W_n}\hspace{-0.16pt}=\hspace{-0.15pt}\sum_{j=1}^nu_j\hspace{1pt}\varphi_j$ and $v_{\mathscr W_n}\hspace{-0.16pt}=\hspace{-0.15pt}\sum_{j=1}^nv_j\hspace{1pt}\varphi_j$ for unique vectors $\mathbf{u}=(u_1,\ldots,u_n)^T$ and $\mathbf{v}=(v_1,\ldots,v_n)^T$. Moreover, by linearity, the equations in \eqref{toy-weak-g} are satisfied for all $w\in\mathscr W_n$ if and only if they are satisfied for the basis functions $\varphi_1,\ldots,\varphi_n$.
Thus,
the computation of $u_{\mathscr W_n},v_{\mathscr W_n}$ (i.e., of $\mathbf{u},\mathbf{v}$) reduces to solving the linear system
\begin{equation*} 
A_{2n}\begin{bmatrix}\mathbf{u}\\\mathbf{v}\end{bmatrix}=\begin{bmatrix}\mathbf{f}\\\mathbf{g}\end{bmatrix},
\end{equation*}
where $\mathbf{f}=\bigl[\int_0^1f(x)\varphi_i(x){\rm d} x\bigr]_{i=1}^n,\ \mathbf{g}=\bigl[\int_0^1g(x)\varphi_i(x){\rm d} x\bigr]_{i=1}^n$ and $A_{2n}$ is the stiffness matrix, which possesses the following saddle point structure:
\[ A_{2n} =\begin{bmatrix}
K_n & H_n\\[3pt]
H_n^T & -\rho M_n
\end{bmatrix}. \]
Here, the blocks $K_n, H_n, M_n$ are square matrices of size $n$, and precisely
{\allowdisplaybreaks\begin{align*}
K_n&=\left[\int_0^1a(x)\varphi_j'(x)\varphi_i'(x){\rm d} x\right]_{i,j=1}^n,\\ 
H_n&=\left[\int_0^1\varphi_j'(x)\varphi_i(x){\rm d} x\right]_{i,j=1}^n=\frac12\begin{bmatrix}
0 & 1 & & & \\
-1 & 0 & 1 & & \\
& \ddots & \ddots & \ddots &\\
& & -1 & 0 & 1 \\
& & & -1 & 0
\end{bmatrix}=-{\rm i}\,T_n(\sin\theta),\\ 
M_n&=\left[\int_0^1\varphi_j(x)\varphi_i(x){\rm d} x\right]_{i,j=1}^n=\frac{h}{6}\begin{bmatrix}
4 & 1 & \\
1 & 4 & 1 & \\
& \ddots & \ddots & \ddots \\
& & 1 & 4 & 1\\
& & & 1 & 4
\end{bmatrix}=\frac{h}{3}\,T_n(2+\cos\theta). 
\end{align*}}%
Note that $K_n$ is exactly the matrix appearing in \eqref{Kn(a).G}. Note also that the matrices $K_n,\,M_n$ are symmetric, while $H_n$ is skew-symmetric: $H_n^T=-H_n={\rm i}\,T_n(\sin\theta)$.

\paragraph{GLT spectral analysis of the Schur complements of the FE discretization matrices}
Assume that the matrices $K_n$ are invertible. This is satisfied, for example, if $a>0$ a.e., in which case the matrices $K_n$ are symmetric positive definite. 
Indeed, assuming that $a>0$ a.e., for every $\xx\in\mathbb R^n\setminus\{\mathbf0\}$ we have
\begin{align*}
\xx^TK_n\xx&=\sum_{i,j=1}^n(K_n)_{ij}x_ix_j=\sum_{i,j=1}^nx_ix_j\int_0^1a(x)\varphi_j'(x)\varphi_i'(x){\rm d}x=\int_0^1a(x)\Biggl(\sum_{i=1}^nx_i\varphi_i'(x)\Biggr)^2{\rm d}x>0,
\end{align*}
where the latter inequality is due to the fact that the function inside the integral is nonnegative and not a.e.\ equal to $0$, because otherwise we would have $\sum_{i=1}^nx_i\varphi_i'(x)=0$ for almost every $x\in[0,1]$ and
\[ \sum_{i=1}^nx_i\varphi_i(x)=\sum_{i=1}^nx_i\varphi_i(0)+\int_0^x\Biggl(\sum_{i=1}^nx_i\varphi_i'(y)\Biggr){\rm d}y=0,\qquad\forall\,x\in[0,1], \]
which is impossible, since $\varphi_1,\ldots,\varphi_n$ is a basis of $\mathscr W_n$ and $\xx\ne\mathbf0$. 
The (negative) Schur complement of $A_{2n}$ is the symmetric matrix given by
\begin{equation}\label{schur-toy}
S_n=\rho M_n+H_n^T\,K_n^{-1}\,H_n=\frac{\rho h}3\,T_n(2+\cos\theta)+T_n(\sin\theta)\,K_n^{-1}\,T_n(\sin\theta).
\end{equation}
In the following theorem, we perform the GLT spectral analysis of the sequence of normalized Schur complements $\{(n+1)S_n\}_n$, and we compute its 
spectral and singular value distribution under the additional necessary assumption that $a\ne0$ a.e.

\begin{theorem}\label{FE_T2}
Let $a:[0,1]\to\mathbb R$ be in $L^1([0,1])$ and let $\rho\in\mathbb R$. Suppose that the matrices $K_n$ are invertible and $a\ne0$ a.e. Then
\begin{equation}\label{FEsGLT}
\{(n+1)S_n\}_n\sim_{\rm GLT}\varsigma(x,\theta)
\end{equation}
and
\begin{equation}\label{FEssigla}
\{(n+1)S_n\}_n\sim_{\sigma,\lambda}\varsigma(x,\theta),
\end{equation}
where
\[ \varsigma(x,\theta)=\frac\rho3(2+\cos\theta)+\frac{\sin^2\theta}{a(x)(2-2\cos\theta)}. \]
\end{theorem}
\begin{proof}
In view of \eqref{schur-toy}, we have
\[ (n+1)S_n=\frac{\rho}3\,T_n(2+\cos\theta)+T_n(\sin\theta)\,\Bigl(\frac1{n+1}K_n\Bigr)^{-1}\,T_n(\sin\theta). \]
Moreover, by \eqref{L->K},
\[ \Bigl\{\frac1{n+1}K_n\Bigr\}_n=\Bigl\{\frac1{n+1}K_n(a)\Bigr\}_n\sim_{\rm GLT}a(x)(2-2\cos\theta). \]
Therefore, under the assumption that $a\ne0$ a.e., the GLT relation \eqref{FEsGLT} follows from {\bf GLT3}--{\bf GLT5}. The singular value and spectral distributions in \eqref{FEssigla} follow from \eqref{FEsGLT} and {\bf GLT1} as the Schur complements $S_n$ are symmetric.
\end{proof}

\subsubsection{FE Discretization of Second-Order Eigenvalue Problems}\label{FE_eigp}
Consider the following second-order eigenvalue problem: find eigenvalues $\lambda_k$ and eigenfunctions $u_k$, for $k=1,2,\ldots,$ such that
\begin{equation}\label{eig-p}
\left\{\begin{aligned}
&-(a(x)u_k'(x))'=\lambda_kc(x)u_k(x), &&\quad x\in(0,1),\\[3pt]
&u_k(0)=u_k(1)=0.
\end{aligned}\right.
\end{equation}

\paragraph{FE discretization}
We consider the approximation of \eqref{eig-p} by linear FEs on a uniform mesh in $[0,1]$ with stepsize $h=\frac1{n+1}$. Let us describe it shortly.
Multiplying both sides of the DE in \eqref{eig-p} by a generic test function $w\in H^1_0([0,1])$ and integrating (formally) over $[0,1]$, we obtain the weak form of \eqref{eig-p}, which
reads as follows: find eigenvalues $\lambda_k$ and eigenfunctions $u_k\in H^1_0([0,1])$, for $k=1,2,\ldots,$ such that
\begin{equation}\label{weak-eig-p}
\int_0^1a(x)u_k'(x)w'(x){\rm d}x=\lambda_k\int_0^1c(x)u_k(x)w(x){\rm d}x,\qquad\forall\,w\in H^1_0([0,1]).
\end{equation}
Let $n\in\mathbb N$, $h=\frac1{n+1}$ and $x_i=ih$, $i=0,\ldots,n+1$. In the linear FE approach based on the mesh $\{x_0,\ldots,x_{n+1}\}$, we fix the subspace $\mathscr W_n=\textup{span}(\varphi_1,\ldots,\varphi_n)\subset H^1_0([0,1])$, where $\varphi_1,\ldots,\varphi_n$ are the hat-functions in \eqref{hat_funs.G} (see also Figure~\ref{hat-funs}). Then, we look for approximations of the exact eigenpairs $(\lambda_k,u_k)$, $k=1,2,\ldots,$ by solving the following (Galerkin) problem: find numerical eigenvalues $\lambda_{k,n}$ and numerical eigenfunctions $u_{k,\mathscr W_n}\in\mathscr W_n$, for $k=1,2,\ldots,n,$ such that
\begin{equation}\label{G-p}
\int_0^1a(x)u_{k,\mathscr W_n}'(x)w'(x){\rm d}x=\lambda_{k,n}\int_0^1c(x)u_{k,\mathscr W_n}(x)w(x){\rm d}x,\qquad\forall\,w\in\mathscr W_n.
\end{equation}
Assuming that the exact and numerical eigenvalues are arranged in ascending order,\,\footnote{\,It can be shown that, in the cases of interest in physical applications, we have $a,c>0$ and the exact eigenpairs $(\lambda_k,u_k)$, $k=1,2,\ldots,$ can be arranged in a sequence such that $0<\lambda_1\le\lambda_2\le\ldots$ and $\lambda_k\to\infty$ as $k\to\infty$. For example, if $a=c=1$ identically then $\lambda_k=k^2\pi^2$ and $u_k(x)=\sin(k\pi x)$ for $k=1,2,\ldots$} 
the pair $(\lambda_{k,n},u_{k,\mathscr W_n})$ is taken as an approximation to the pair $(\lambda_k,u_k)$ for $k=1,\ldots,n$.
Since $\{\varphi_1,\ldots,\varphi_n\}$ is a basis of $\mathscr W_n$, we can write $u_{k,\mathscr W_n}=\sum_{j=1}^nu_{k,j}\varphi_j$ for a unique vector $\mathbf u_{k,n}=(u_{k,1},\ldots,u_{k,n})^T$. Moreover, by linearity, the equation in \eqref{G-p} is satisfied for all $w\in\mathscr W_n$ if and only if it is satisfied for the basis functions $\varphi_1,\ldots,\varphi_n$. Thus, the computation of $(\lambda_{k,n},u_{k,n})$ (i.e., of $(\lambda_{k,n},\mathbf u_{k,n})$) reduces to 
solving the generalized eigenvalue problem
\begin{equation}\label{gen-eig-p}
K_n\mathbf u_{k,n}=\lambda_{k,n}M_n\mathbf u_{k,n},
\end{equation}
where 
\[ K_n=\left[\int_0^1a(x)\varphi_j'(x)\varphi_i'(x){\rm d}x\right]_{i,j=1}^n,\qquad M_n=\left[\int_0^1c(x)\varphi_j(x)\varphi_i(x){\rm d}x\right]_{i,j=1}^n. \]
The symmetric matrices $K_n$ and $M_n$ are referred to as the stiffness and mass matrices, respectively. Assuming that $M_n$ is invertible, it is clear from \eqref{gen-eig-p} that the numerical eigenvalues $\lambda_{k,n}$, $k=1,\ldots,n,$ are just the eigenvalues of the matrix
\begin{equation}\label{L}
L_n=M_n^{-1}K_n.
\end{equation}

\paragraph{GLT spectral analysis of the FE discretization matrices}
Using the theory of GLT sequences, we now derive the spectral and singular value distribution of the normalized matrix-sequence $\{(n+1)^{-2}L_n\}_n$ under very weak hypotheses on the DE coefficients $a,c$.
Note that $L_n$ is well-defined whenever $M_n$ is invertible. In particular, $L_n$ is well-defined if $c>0$ a.e., because in this case $M_n$ is symmetric positive definite. Indeed, assuming that $c>0$ a.e., for every $\xx\in\mathbb R^n\setminus\{\mathbf0\}$ we have
\begin{align*}
\xx^TM_n\xx&=\sum_{i,j=1}^n(M_n)_{ij}x_ix_j=\sum_{i,j=1}^nx_ix_j\int_0^1c(x)\varphi_j(x)\varphi_i(x){\rm d}x=\int_0^1c(x)\Biggl(\sum_{i=1}^nx_i\varphi_i(x)\Biggr)^2{\rm d}x>0,
\end{align*}
where the latter inequality is due to the fact that the function inside the integral is nonnegative and not a.e.\ equal to $0$. 

\begin{theorem}
If $a,c:[0,1]\to\mathbb R$ belong to $L^1([0,1])$ and $c>0$ a.e.\ then
\begin{equation}\label{L-glt}
\{(n+1)^{-2}L_n\}_n\sim_{\rm GLT}\frac{a(x)}{c(x)}\,\frac{6-6\cos\theta}{2+\cos\theta}
\end{equation}
and
\begin{equation}\label{L-sigla}
\{(n+1)^{-2}L_n\}_n\sim_{\sigma,\lambda}\frac{a(x)}{c(x)}\,\frac{6-6\cos\theta}{2+\cos\theta}.
\end{equation}
\end{theorem}
\begin{proof}
We note that $K_n=K_n(a)$ and $M_n=M_n(c)$, where $K_n(g)$ and $M_n(g)$ are the matrices appearing in \eqref{L->K} and \eqref{Mn(g)}, respectively.
Hence, by \eqref{L->K} and \eqref{Mn(g)}, we have
\[ \{(n+1)^{-1}K_n\}_n\sim_{\rm GLT}a(x)(2-2\cos\theta),\qquad\{(n+1)M_n\}_n\sim_{\rm GLT}c(x)\Bigl(\frac23+\frac13\cos\theta\Bigr). \]
Moreover, we observe that
\[ (n+1)^{-2}L_n=((n+1)M_n)^{-1}((n+1)^{-1}K_n). \]
Thus, the relations \eqref{L-glt}--\eqref{L-sigla} follow from Theorem~\ref{exe-preconditioning}, taking into account that the matrices $M_n$ are symmetric positive definite and $c(x)(\frac23+\frac13\cos\theta)\ne0$ a.e.\ because of our assumption that $c>0$ a.e.
\end{proof}

\section*{Acknowledgements}
The author is member of the Research Group GNCS (Gruppo Nazionale per il Calcolo Scientifico) of INdAM (Istituto Nazionale di Alta Matematica).
This work was supported by the Department of Mathematics of the University of Rome Tor Vergata through the MUR Excellence Department Project MatMod@TOV (CUP E83C23000330006).

\end{document}